\documentclass[reqno]{amsart}

\usepackage{amssymb,amsfonts,amsmath,amsthm}
\usepackage{mathrsfs}
\usepackage{graphicx}
\usepackage{color}
\usepackage{verbatim}

\newtheorem{thm}{Theorem}[section]
\newtheorem{pro}[thm]{Proposition}
\newtheorem{lem}[thm]{Lemma}
\newtheorem{cor}[thm]{Corollary}

\theoremstyle{definition}

\newtheorem*{ackn}{Acknowledgements}

\theoremstyle{remark}
\newtheorem{rmk}[thm]{Remark}
\newtheorem*{cla}{Claim}
\newtheorem{prb}{Problem}
\newtheorem{que}[prb]{Question}

\numberwithin{equation}{section}

\def\J{\mathscr{J}}
\def\D{\mathscr{D}}
\def\R{\mathscr{R}}
\def\L{\mathscr{L}}
\def\H{\mathscr{H}}
\def\cE{\mathcal{E}}

\def\es{\varnothing}

\def\lnor{\trianglelefteq}
\def\ol#1{\overline{#1}}

\def\ig#1{\mathsf{IG}(#1)}

\def\pre#1#2{\langle #1 \; | \; #2 \rangle}

\def\Sch{Sch\"u\-tzen\-ber\-ger }

\begin{document}


\title[Free idempotent generated semigroups]%
{Free idempotent generated semigroups: \\ The word problem and structure via gain graphs} 


\author{IGOR DOLINKA}

\address{Department of Mathematics and Informatics, University of Novi Sad, Trg Dositeja Obradovi\'ca 4,
21101 Novi Sad, Serbia}

\email{dockie@dmi.uns.ac.rs}









\begin{abstract}
Building on the previous extensive study of Yang, Gould and the present author, we provide a more precise insight into the 
group-theoretical ramifications of the word problem for free idempotent generated semigroups over finite biordered sets. We 
prove that such word problems are in fact equivalent to the problem of computing intersections of cosets of certain subgroups
of direct products of maximal subgroups of the free idempotent generated semigroup in question, thus providing decidability
of those word problems under group-theoretical assumptions related to the Howson property and the coset intersection property.
We also provide a basic sketch of the global semigroup-theoretical structure of an arbitrary free idempotent generated semigroup,
including the characterisation of Green's relations and the key parameters of non-regular $\D$-classes. In particular, we prove
that all \Sch groups of $\ig{\cE}$ for a finite biordered set $\cE$ must be among the divisors of the maximal subgroups of $\ig{\cE}$.
\end{abstract}


\maketitle



\section{Introduction}\label{sec:intro}

Let $S$ be a semigroup, and let us denote by $E=E(S)$ the set of its idempotents. A common way to record the 
``idempotent structure'' of $S$ is to consider the partial algebra $\cE_S=(E,\cdot)$ obtained by retaining only the products
$ef$ such that $\{ef,fe\}\cap\{e,f\} \neq \es$. (Note that if, for example, $ef\in\{e,f\}$ then $fe$ also must belong to $E$,
so either both products are recorded or none of them. In the former case we say that $\{e,f\}$ is a \emph{basic pair}.) 
With such a partial algebra at hand, one can naturally introduce two quasi-orders $\leq_\ell$ and $\leq_r$ on $E$ given by 
$e\leq_\ell f$ if and only if $ef=e$, and $e\leq_r f$ if and only if $fe=e$; furthermore, it is easily seen that the relation 
$\leq=\leq_\ell\cap\leq_r$ is a partial order on $E$ and this is the so-called \emph{natural order} \cite{HoBook} on the set of 
idempotents of the semigroup $S$. Because of the presence of these two quasi-orders, the partial algebra $\cE_S$ is called
the \emph{biordered set} \cite{Na,Ea2} of the semigroup $S$. We immediately note that the biordered set of a finite semigroup
is of course finite, but the converse is not true: there are finite biordered sets that do not arise from any finite semigroup 
\cite{Ea1}.

In his extensive study \cite{Na}, Nambooripad provided an axiomatic approach to abstract biordered sets, and his axioms were
proved to be complete by Easdown \cite{Ea2} in the sense that a partial algebra $\cE=(E,\cdot)$ satisfies these axioms if and
only if it is isomorphic to a ``concrete'' biordered set $\cE_S$ of some semigroup $S$, one of the most striking aspects being
that the axiomatisation is actually finite. Hence, biordered sets serve as a convenient vehicle to record quite an amount
of intrinsic structural (and even ``geometric'', in a sense) information about a semigroup $S$, especially if the latter is 
idempotent generated, that is, if $S=\langle E(S) \rangle$. Idempotent generated semigroups abound in algebra, and therefore
constitute a very important area worth a detailed study. Let us just mention the classical result of Howie \cite{Ho} who proved
that the full semigroups of singular transformations on a finite set are idempotent generated; an analogous result holds for
the semigroup of all singular matrices over a field (Erdos \cite{Er}), which was later expanded to matrices over
division rings (Laffey \cite{La}). Inspired by this, Putcha \cite{Pu} characterised all idempotent generated reductive linear 
algebraic monoids \cite{PuBook,ReBook}.

To gain a deeper insight into the general structure of idempotent generated semigroups, one introduces the \emph{free idempotent
generated semigroup over a biordered set} $\cE$, given by the presentation 
$$
\ig{\cE} = \pre{\ol{E}}{\ol{e}\ol{f}=\ol{e\cdot f}\text{ whenever }\{e,f\}\text{ is a basic pair in }\cE}.
$$
(Here, just as in \cite{YDG2}, the bar notation serves the purpose of distinguishing between the actual elements of $\cE$ and
their products, and, on the other hand, the elements of $\ig{\cE}$ which are equivalence classes consisting of words over $E$
taken as an alphabet. So, for example, $\ol{e}\ol{f}$ denotes a multiplication happening in $\ig{\cE}$, while $e\cdot f$ refers
to a product in $\cE$, with $\ol{e\cdot f}$ being the corresponding generating element in $\ig{\cE}$.) Loosely speaking, $\ig{\cE}$
is the ``free-est'' semigroup with biordered set $\cE$; more precisely, it is the ``semigroup part'' of the free (initial) object in 
the category whose objects are pairs $(S,\phi)$ where $S$ is a semigroup and $\phi:\cE\to\cE_S$ is an isomorphism of biordered 
sets, while the morphisms from $(S,\phi)$ to $(T,\psi)$ are semigroup homomorphisms $\theta:S\to T$ such that $\phi\theta=\psi$ 
(implying that the restriction of $\theta$ to $E(S)$ is a biordered set isomorphism $\cE_S\to\cE_T$). 

For a long time, investigations of free idempotent generated semigroups were centred around their maximal subgroups. At least 
from the 80s of the last century, a folklore conjecture was in circulation stating that the all maximal subgroups of $\ig{\cE}$ 
must be free for any biorder $\cE$. We refer to \cite{NP,Pa1,Pa2} for a sampler of several partial results in this vein, as well
as \cite{McE}, where the conjecture was ``officially'' recorded. After being dormant for some time, the topic was greatly revived
by the seminal paper of Brittenham, Margolis and Meakin \cite{BMM}, who eventually refuted the freeness conjecture (they found 
a 72-element semigroup $S$ such that $\ig{\cE_S}$ has $\mathbb{Z}\oplus\mathbb{Z}$ among its maximal subgroups), and in fact 
showed that the study of the structure of $\ig{\cE}$ legitimately belongs, besides abstract algebra, to the realm of geometry 
and algebraic topology. Namely, they proved that the maximal subgroups of $\ig{\cE}$ are in fact fundamental groups of certain 
cell 2-complexes constructed from the biorder $\cE$.

However, the real turning point came with the paper of Gray and Ru\v skuc \cite{GR1} who used Reidemeister-Schreier-type rewriting
methods for subgroups of monoids \cite{Ru} to obtain explicit presentations for maximal subgroups of $\ig{\cE}$. Applying this,
they were able to show that the complete opposite of the former freeness conjecture is true: \emph{every} group can arise as a
maximal subgroup of $\ig{\cE}$ for a suitably chosen biorder $\cE$. Later, the author and Ru\v skuc proved \cite{DR} that biordered
sets of idempotent semigroups (bands) suffice for this aim, and that all finitely presented groups can be realised by using only biorders 
of finite bands. A number of papers followed suit with the aim of computing the maximal subgroups of $\ig{\cE}$ for biordered sets 
of some of the most ``popular'' finite semigroups, such as the transformation monoids \cite{GR2} and partial transformation monoids
\cite{Do2} on finite sets, and full linear monoids over division rings \cite{DG}.

The shift of the focus of research in this area came with the paper \cite{DGR}, which was the first to study the word problem for 
semigroups $\ig{\cE}$ and their structural aspects other than the maximal subgroups. It was shown there that, given a finite biorder
$\cE$, there exists an algorithm deciding whether a given word represents a regular element of $\ig{\cE}$, and, furthermore, if the
word problem is decidable in the maximal subgroups of $\ig{\cE}$, then there exists an algorithm which, given two words representing 
regular elements, decides their equality in $\ig{\cE}$. On the other hand, a biordered set $\cE$ arising from a finite semigroup was 
exhibited such that the word problem for $\ig{\cE}$ is undecidable (even though all subgroups of $\ig{\cE}$ have decidable word problems). 
Also, it was shown that any group can arise as a \Sch group of a non-regular $\D$-class in $\ig{\cE}$.

This paper is a direct sequel to \cite{YDG2}, the most recent contribution pertaining to the word problem and structure of semigroups
of the form $\ig{\cE}$. It was shown there that in the case of a finite biorder $\cE$, the word problem for $\ig{\cE}$ is in fact
equivalent to an algorithmic problem in group theory (thus completely eliminating the original semigroup-theoretical setting of 
the problem), namely, to a specific class of CSPs (Constraint Satisfaction Problems) with respect to certain rational subsets of
direct products of maximal subgroups of $\ig{\cE}$ (with both the subsets and the groups being algorithmically computable from $\cE$).
This provided a way to prove that $\ig{\cE}$ has decidable word problem for a wide class of finite biorders $\cE$, including those of
transformation monoids $T_n$ and partial transformation monoids $PT_n$; also, it implied structural features for
$\ig{\cE}$ such as being weakly abundant (aka Fountain) \cite{Lawson,MS}. Here, we first use several remarks from algebraic graph theory
(about the so-called gain graphs) to show in Section \ref{sec:gain} that the rational subsets of direct products just mentioned are 
in fact (right) cosets of certain finitely generated subgroups which are effectively computable for finite $\cE$. This immediately
makes it possible to prove decidability results for the word problem of $\ig{\cE}$ for an even wider class of finite biorders $\cE$
related to the Howson property and the coset intersection property for direct products of maximal subgroups of $\ig{\cE}$. In 
Section \ref{sec:str} we use the this machinery to discover the basics of the general 
semigroup-theoretical structure of $\ig{\cE}$ for finite $\cE$ such as the characterisation of the Green's relations, the shape and
size of the non-regular $\D$-classes, etc. In particular, we prove a somewhat unexpected result: every \Sch group of $\ig{\cE}$
must divide (= be a quotient of a subgroup) some maximal subgroup of $\ig{\cE}$, meaning that \Sch groups of $\ig{\cE}$ are not
completely independent of their maximal subgroups (contrary to the conjecture that was suggested in the final section of \cite{DGR}, 
albeit in an implicit way). We finish by posing several problems for future investigations.


\section{Preliminaries}\label{sec:prelim}

In this section we gather, in a succinct fashion, all the necessary prerequisites for the remaining sections. This includes both
basics of semigroup theory, and also a summary of the pertinent concepts and results from \cite{YDG2}.

\subsection{Basic notions of semigroup theory}\label{subsec:basic}

Our basic reference for semigroup theory is \cite{HoBook}; here we give an abridged exposition of the main concepts we need,
referring to this monograph for details.

Perhaps the most fundamental tool in dealing with the structure of semigroups are \emph{Green's relations} $\R,\L,\H\,\D,\J$
classifying the elements according to the right/left/two-sided (principal) ideals they generate. For a semigroup $S$ and $a,b\in S$ 
we define:
$$
a\;\R\; b \Leftrightarrow aS^1=bS^1, \quad a\;\L\; b \Leftrightarrow S^1a=S^1b,\quad a\;\J\; b \Leftrightarrow
S^1aS^1=S^1bS^1,
$$
where $S^1$ denotes $S$ with an identity element adjoined (unless $S$ already has one). We let $\H=\R\cap\L$ and let $\D=\R\vee\L$
be the least equivalence on $S$ containing both $\R$ and $\L$. It is well known that $\R\circ\L=\L\circ\R$ holds in
any semigroup, so $\D$ coincides with the latter relational composition. Also, $\R$ is always a left congruence, while $\L$
is a right congruence. In general we have $\H\subseteq \R,\L\subseteq\D\subseteq\J$. There is a large class of semigroups
for which $\D=\J$, including all finite (and indeed all periodic) semigroups. For $\mathscr{K}\in\{\H,\R,\L,\D,\J\}$ we denote 
the $\mathscr{K}$-class containing $a\in S$ by $K_a$ (where $K\in\{H,R,L,D,J\}$, respectively). It follows from \cite[Lemma 2.4]{YDG2}
that $D_a=J_a$ whenever $a$ is a regular element of $\ig{\cE}$, where $\cE$ is a finite biordered set (see below for the 
semigroup-theoretical notion of regularity). Later in this paper we shall see that in fact $\D=\J$ holds for any such free
idempotent generated semigroup $\ig{\cE}$.

For a biordered set $\cE$ we abuse the notation just introduced and denote by $\R$ and $\L$ the equivalence relations induced by the 
quasi-orders $\leq_r$ and $\leq_\ell$, respectively. However, it is quite straightforward to show that for any semigroup $S$ we
have $e\;\R\;f$ in $\cE_S$ if and only if $e\;\R\;f$ in $S$, and also $e\;\L\;f$ in $\cE_S$ if and only if $e\;\L\;f$ in $S$.
Therefore, bearing in mind that every abstract biordered set is essentially a biordered set of some semigroup \cite{Ea2}, there is
no harm in making this notational blur between these equivalences in biorders and semigroups. Furthermore, if we define in biordered
sets a third equivalence $\D=\R\vee\L$, and if $S$ is idempotent generated, then the previous remark supporting the identification of
the $\D$ relations in $\cE_S$ and $S$ also holds by the results from \cite{DGR,FG}. In other words, for idempotent generated semigroups
$S$ we have $e\;\D\;f$ in $\cE_S$ if and only if $e\;\D\;f$ in $S$. However, we stress that in general biordered sets fail to satisfy
$\D=\R\circ\L=\L\circ\R$.

There is a natural partial order $\leq$ defined on the set of all $\J$-classes of a semigroup $S$ given by $J_a\leq J_b$ if and only if
$S^1aS^1\subseteq S^1bS^1$. (In a similar way, one can partially order the sets of $\R$-classes and of $\L$-classes of a semigroup,
respectively.) Clearly, the partial order $\leq$ induces a quasi-order on the set of all $\D$-classes of $S$ (which then turns into 
a partial order when $\D=\J$).

An element $a\in S$ is \emph{regular} if there exists $b\in S$ such that $aba=a$. An equivalent statement of this property is that $a$
has an \emph{inverse}: and element $a'\in S$ such that $aa'a=a$ and $a'aa'=a'$ (since $aba=a$ implies that $bab$ is an inverse of $a$).
Any $\D$-class of a semigroup consists either entirely of regular or non-regular elements, and in this sense we speak of \emph{regular} 
and \emph{non-regular $\D$-classes}, respectively. In fact, regular $\D$-classes coincide with $\D$-classes containing idempotents.
Furthermore, for each idempotent $e$ its $\H$-class $H_e$ is a group with identity $e$; this is a maximal subgroup of $S$,
and all maximal subgroups of $S$ arise in this way. Any two maximal subgroups belonging to a single (regular) $\D$-class
must be isomorphic, so that the maximal subgroup is an invariant of a regular $\D$-class up to isomorphism.

A semigroup $S$ is called \emph{simple} if its only ideal is $S$ itself, and a semigroup with zero $0\in S$ is called \emph{0-simple} 
if $\{0\}$ and $S$ are its only ideals and $S^2\neq\{0\}$ (that is, the multiplication in $S$ is not null). A (0-)simple semigroup is
\emph{completely (0-)simple} if it has a \emph{primitive idempotent}, an idempotent which is minimal (in the natural order defined above) 
among the non-zero ones. There is a particularly nice and handy description of completely 0-simple semigroups provided by the 
foundational result called the \emph{Rees-Su\v skevi\v c Theorem}: a semigroup $S$ is completely 0-simple if and only if it is
isomorphic to a semigroup of the form $\mathcal{M}^0[G;I,\Lambda;P]$ (called the \emph{Rees matrix semigroup}) described below.
Here, $G$ is a group, $I,\Lambda$ are two index sets (labelling the $\R$-classes and the $\L$-classes, respectively), and the 
\emph{sandwich matrix} $P=(p_{\lambda i})$ is a $\Lambda\times I$ matrix over $G\cup\{0\}$ in which every row and every column contains 
at least one non-zero entry. Elements of $\mathcal{M}^0[G;I,\Lambda;P]$ are the zero element $0$ and triples $(i,g,\lambda)$ such that 
$i\in I$, $g\in G$ and $\lambda\in\Lambda$. For the multiplication of these elements we have the rules $00=0(i,g,\lambda)=(i,g,\lambda)0=0$ and
$$
(i,g,\lambda)(j,h,\mu) = \left\{\begin{array}{ll}
(i,gp_{\lambda j}h,\mu) & \text{if }p_{\lambda j}\neq 0,\\
0 & \text{otherwise.}
\end{array}\right.
$$
Arguably one of the most important ways in which completely 0-simple semigroups arise in arbitrary semigroups is via their \emph{principal
factors}. Namely, if $J$ is a $\J$-class of a semigroup $S$ then we define its principal factor $\ol{J}=J\cup\{0\}$ with multiplication
$$
x\cdot y = \left\{\begin{array}{ll}
xy & \text{if }x,y,xy\in J,\\
0 & \text{otherwise.}
\end{array}\right.
$$
Now $\ol{J}$ is either $0$-simple or has null multiplication otherwise. If $S$ has finitely many 
idempotents then, as explained in \cite[Lemma 2.4]{YDG2}, if $J$ contains idempotents then it consists of only one (regular) $\D$-class, 
and the corresponding principal factor is completely $0$-simple. In particular, this is the case with the $\D$-classes of $\ig{\cE}$ 
when $\cE$ is a finite biorder; we will return to this important remark shortly.

\subsection{$\ig{\cE}$, part one: the basic properties}\label{subsec:ige1}

We provide a list of some known facts about free idempotent generated semigroups $\ig{\cE}$. Throughout these statements we assume that 
the semigroup $S$ is idempotent generated, $S=\langle E(S)\rangle$.
\begin{itemize}
\item[(1)] $E(\ig{\cE})=\{\ol{e}:\ e\in E\}$, so that the biordered set of $\ig{\cE}$ is isomorphic to $\cE$ (\cite{Ea2}).
\item[(2)] The map $\ol{e}\mapsto e$ extends to a surjective homomorphism $\varphi:\ig{\cE_S}\to S$ (this was noted in \cite{GR1}).
\item[(3)] For any $e\in E$, $\varphi$ maps the $\R$-class ($\L$-class) of $\ol{e}$ in $\ig{\cE_S}$ onto the $\R$-class (resp.\ 
$\L$-class) of $e$ in $S$ (this follows from \cite{FG}).
\item[(4)] Consequently, given a regular $\D$-class $D=D_{\ol{e}}$ in $\ig{\cE_S}$ there is a bijection between the set of $\R$-classes
($\L$-classes) of $D$ and the corresponding set in the $\D$-class $D_e$ in $S$. 
\end{itemize}
Therefore, given a (finite) biordered set $\cE$, we already know a substantial amount of information about the regular $\D$-classes that
$\ig{\cE}$ will have. In more detail, if $D_1,\dots,D_m$ is the list of all $\D$-classes of $\cE$, then $\ig{\cE}$ will also have precisely 
$m$ regular $\D$-classes, say $D'_1,\dots, D'_m$, and the ``shapes'' of corresponding $\D$-classes will be the same: for $1\leq k\leq m$, 
if $I_k$ is an index set enumerating $D_k/\R$ and $\Lambda_k$ enumerates $D_k/\L$, then the number of $\R$-($\L$-)classes in $D'_k$
will be $|I_k|$ (resp.\ $|\Lambda_k|$). For this reason, there is no harm in slightly abusing notation and assuming that $I_k$ and 
$\Lambda_k$ also enumerate the $\R$-classes (resp.\ $\L$-classes) of $D'_k$. In addition, we have already remarked that each $D'_k$ will 
coincide with its $\J$-class and so the $\J$-order imposes a partial order on the regular $\D$-classes of $\ig{\cE}$.

Let us also note at this point that the explicit presentations for maximal subgroups in regular $\D$-classes of $\ig{\cE}$ (for 
arbitrary biorders $\cE$) were found by Gray and Ru\v skuc in \cite{GR1}. It was later shown in \cite[Theorem 3.10]{DGR} that there
is actually an algorithm which takes a finite biordered set $\cE$ as input and outputs finite presentations for maximal subgroups
in all regular $\D$-classes of $\ig{\cE}$.

We now recall some of the main results of \cite{DGR} implying that if $\cE$ is finite then there exists an algorithm deciding if the given
word from $E^+$ represents a regular element in $\ig{\cE}$; furthermore, if all maximal subgroups of $\ig{\cE}$ have decidable word problems,
there exists an algorithm which, given $u,v\in E^+$ representing regular elements of $\ig{\cE}$, decides whether $\ol{u}=\ol{v}$ (this is
what we mean by the phrase that the `regular part of the word problem is decidable').

\begin{thm}\label{thm:DGR}
\begin{itemize}
\item[(i)] For any $w\in E^+$, $\ol{w}$ is a regular element of $\ig{\cE}$ if and only if $w$ can be factorised as
$$
w = uev
$$
such that $e\in E$ and $\ol{ue}\;\L\;\ol{e}\;\R\;\ol{ev}$. In such a case, $\ol{e}\;\D\;\ol{w}$.
\item[(ii)] If $\cE$ is finite, then there exists an algorithm which, given $e\in E$ and $v\in E^\ast$, decides whether
$\ol{ev}\;\R\;\ol{e}$ holds in $\ig{\cE}$ and if so, returns an $f\in E$ such that $\ol{ev}\;\L\;\ol{f}$.
Dually, there exists an algorithm which, given $e\in E$ and $u\in E^\ast$, decides whether
$\ol{ue}\;\L\;\ol{e}$ holds in $\ig{\cE}$ and if so, returns a $g\in E$ such that $\ol{ue}\;\R\;\ol{g}$.
Consequently, there exists an algorithm which, given $w\in E^+$, decides whether $\ol{w}$ is a regular
element of $\ig{\cE}$, and if so, returns $e,f\in E$ such that $\ol{e}\;\R\;\ol{w}\;\L\;\ol{f}$.
\end{itemize}
\end{thm}

As remarked in \cite[Remark 2.6]{YDG2}, if $\ol{e}\;\R\;\ol{ev}$ then for any prefix $v'$ of $v$ and any word $z\in E^*$ we have 
$\ol{ze}\;\R\;\ol{zev'}$ (analogously, $\ol{e}\;\L\;\ol{ue}$ implies $\ol{ez}\;\L\;\ol{u'ez}$ for any suffix $u'$ of $u$ and arbitrary
$z\in E^*$). In particular, if the factorisation $w=uev$ is as in (i) above then $\ol{ue}\;\R\;\ol{w}\;\L\;\ol{ev}$.
The second part of this remark shows that the converse of the statement (i) from the previous theorem holds for finite
biorders, and the relevance of this in the present paper makes it worth recording as a separate statement.

\begin{lem}\label{lem:conv}
Let $\cE$ be a finite biordered set. If $w\in E^+$ has a factorisation of the form $w=uev$ such that $\ol{e}\;\D\;\ol{w}$ then $\ol{w}$ 
represents a regular element of $\ig{\cE}$ and we have $\ol{ue}\;\L\;\ol{e}\;\R\;\ol{ev}$.
\end{lem}

A letter $e$ with the properties from the previous lemma (i.e.\ from Theorem \ref{thm:DGR}(i)) will be called the \emph{seed} of $w$;
so, succinctly, words representing regular elements are precisely those having seeds. We draw the attention to the fact that the seed
letter (if any) need not to be unique; in fact, the results of \cite{FG} show that it might happen that every letter of a word is a seed.

\subsection{$\ig{\cE}$, part two: the word problem and related concepts}\label{subsec:ige2}

We begin by recalling the key notion of a minimal r-factorisation from \cite{YDG2} and the related ideas and results. For a word $w\in E^+$
we say that the factorisation
$$
w = w_1\dots w_m
$$
is an \emph{r-factorisation} if $\ol{w_i}$ is a regular element of $\ig{\cE}$ for all $1\leq i\leq m$. Since there are only finitely many 
(r-)factorisations of a word, we can spot coarsest such factorisations in the following sense: whenever the element $\ol{w_i\dots w_j}$ is 
regular in $\ig{\cE}$ for some $i\leq j$ then necessarily $i=j$. Such r-factorisations are called \emph{minimal}.

To navigate more easily within a word (over $E$) we define the \emph{position} of a letter. Namely, if $w=uev\in E^+$
where $e\in E$, we say that the position of the indicated occurrence of the letter $e$ is $|ue|=|u|+1$.
Similarly, given a factorisation $w=w_1\dots w_m$ we can naturally associate an increasing sequence of 
integers $(\alpha_1,\dots,\alpha_m)$ by recording the positions of the first (from the left) letters of
subwords $w_1,\dots,w_m$, respectively. Here we have $1=\alpha_1<\alpha_2<\dots<\alpha_m\leq |w|$.

The following is the main result of \cite{YDG2} regarding the minimal r-factorisations of words representing
the same element in $\ig{\cE}$.

\begin{thm}\label{thm:Dfing}
Let $u,v\in E^+$ be such that $\ol{u}=\ol{v}$, and let $u=p_1\dots p_m$ and $v=q_1\dots q_s$ be arbitrary
minimal r-factorisations. Then $m=s$ and for all $1\leq i\leq m$ we have $\ol{p_i}\;\D\;\ol{q_i}$. Furthermore, 
we have $\ol{p_1}\;\R\;\ol{q_1}$ and $\ol{p_m}\;\L\;\ol{q_m}$.
\end{thm}

This result gives rise to the notion of the \emph{$\D$-fingerprint} as an invariant of an element of $\ig{\cE}$:
namely, the previous result states that given such an element, any word $w\in E^+$ representing that element,
and any minimal r-factorisation $w=w_1\dots w_m$, the sequence of regular $\D$-classes 
$$
(D_{\ol{w_1}},\dots,D_{\ol{w_m}})
$$
of $\ig{\cE}$ is fixed, in the sense that it only depends only on the element of $\ig{\cE}$ we have started with.

From this point on, all the way through the remainder of this article, all considered biordered sets will be 
assumed to be finite (unless explicitly state otherwise). Therefore, if $(D_1,\dots,D_m)$ is the $\D$-fingerprint 
of $\mathbf{x}\in\ig{\cE}$ then each of the regular $\D$-classes $D_k$ coincides with its corresponding $\J$-class.
This opens the possibility to use the Rees matrix form of the corresponding principal factors, so that if 
$\mathbf{x}=\mathbf{r}_1\dots\mathbf{r}_m$ (a decomposition arising from a minimal r-factorisation of a word
representing $\mathbf{x}$) with $\mathbf{r}_k\in D_k$ for all $1\leq k\leq m$, a convenient way of writing
the element $\mathbf{r}_k$ is in the form of a triple $(i_k,g_k,\lambda_k)$; here $i_k\in I_k$ and $\lambda_k
\in\Lambda_k$ are coming from the corresponding index sets of $D_k$ and $g_k$ is an element of the (unique up
to isomorphism) maximal subgroup of $D_k$. This means that a typical element of $\ig{\cE}$ with the given 
$\D$-fingerprint would be written in the form 
$$
\mathbf{x} = (i_1,g_1,\lambda_1)\dots(i_m,g_m,\lambda_m).
$$ 
The process of rewriting an element of $\ig{\cE}$ into this form was discussed in detail in \cite[Section 4]{YDG2},
culminating in Theorem 4.3 of that paper which provides an explicit rewriting formula. This theorem, along with 
Proposition 4.1 of \cite{YDG2}, shows that there is an algorithm which, given a word from $E^+$, computes the
corresponding element of $\ig{\cE}$ in the above form of a product of Rees matrix triples. Such a product
representation is in general not unique, so the word problem for $\ig{\cE}$ effectively boils down to the
issue of establishing which products of these triples (of a given $\D$-fingerprint) are equal.

To this end, in \cite[Section 5]{YDG2} a combinatorial structure is introduced, called the \emph{contact automaton}
of two regular $\D$-classes $D_1,D_2$ of $\ig{\cE}$. Namely, it already transpired from \cite[Proposition 4.1]{YDG2}
that for each $e\in E$, the idempotent $\ol{e}$ induces two partial (possibly empty) transformations 
$\sigma_e^{(1)},\sigma_e^{(2)}$ of $I_1,I_2$, respectively, acting from the left on these index sets, as well as two
partial transformations $\tau_e^{(1)},\tau_e^{(2)}$ acting on $\Lambda_1,\Lambda_2$, respectively, from the right.
These partial maps, effectively computable from $\cE$, were instrumental in determining the rewriting of a regular
element of $\ig{\cE}$ into a ``Rees matrix triple form''. Now, the states (vertices) of the contact automaton 
$\mathcal{A}(D_1,D_2)$ are pairs $\Lambda_1\times I_2$ and we have a transition (edge) between pairs $(\lambda,i)$
and $(\mu,j)$ labelled by $e$ if and only if either 
\begin{itemize}
	\item $\lambda = \mu\tau_e^{(1)}$ and $\sigma_e^{(2)}i = j$, or
	\item $\lambda\tau_e^{(1)}=\mu$ and $i = \sigma_e^{(2)}j$.
\end{itemize}
The transitions are two-way (meaning that the edges can be traversed in both directions) but we still provide them
with orientation because each of them are also labelled by an element of the group $G_1\times G_2^\partial$ (where
$G_1,G_2$ are the maximal subgroups of $D_1,D_2$, as computed in \cite{GR1}, and $G^\partial$ denotes the \emph{dual
group} of $G$ with operation $\ast$ defined by $g\ast h=hg$); so, in one direction, the label of the transition
$\mathbf{t}=((\lambda,i),e,(\mu\,j))$ is $\ell(\mathbf{t})=(g_1,g_2)$ for certain elements $g_1\in G_1$ and $g_2\in G_2$
that are easily determined from $\lambda,\mu,i,j$ and the previous partial maps associated with $e$, while in the
opposite direction the label is $\ell(\mathbf{t}^{-1})=\ell(\mathbf{t})^{-1}=(g_1^{-1},g_2^{-1})$.

Arbitrary walks (called somewhat erroneously `paths' in \cite{YDG2}) also have their labels obtained by multiplying
the labels of edges (within the group $G_1\times G_2^\partial$) as we traverse them within the walk (respecting the
orientation). Now let $\rho(\lambda,i;\mu,j)$ denote the set of labels of all walks from $(\lambda,i)$ to $(\mu,j)$
in $\mathcal{A}(D_1,D_2)$. In \cite{YDG2} it is was argued that $\rho(\lambda,i;\mu,j)$ is always a \emph{rational subset}
of the group $G_1\times G_2^\partial$ (recall that a rational subset of a group/monoid is a subset obtained in finitely 
many steps from the finite subsets by means of union, product, and submonoid generation, see \cite{BS,Lo,LS}). Thus we can
now restate the main result of \cite{YDG2} (Theorem 5.2) which characterises the word problem for $\ig{\cE}$ in purely 
group-theoretical terms.

\begin{thm}\label{thm:ig-wp}
Let
$$
\mathbf{u} = (i_1,a_1,\lambda_1)(i_2,a_2,\lambda_2)\dots(i_m,a_m,\lambda_m)
$$
and 
$$
\mathbf{v} = (j_1,b_1,\mu_1)(j_2,b_2,\mu_2)\dots(j_m,b_m,\mu_m)
$$
be two elements of $\ig{\cE}$, where the above decompositions arise from minimal r-factorisations of two words $u,v\in E^+$
such that $\ol{u}=\mathbf{u}$ and $\ol{v}=\mathbf{v}$, both with $\D$-fingerprint $(D_1,D_2,\dots,D_m)$. Then $\mathbf{u}=
\mathbf{v}$ holds in $\ig{\cE}$ if and only if $i_1=j_1$, $\lambda_m=\mu_m$, and there exist $x_r\in G_r$, $2\leq r\leq m-1$,
such that 
\begin{align*}
      (a_1^{-1}b_1, x_2) &\in \rho_1(\lambda_1,i_2;\mu_1,j_2),\\
      (a_2^{-1}x_2^{-1}b_2,x_3) &\in \rho_2(\lambda_2,i_3;\mu_2,j_3),\\
			&\ \vdots\\
			(a_{m-2}^{-1}x_{m-2}^{-1}b_{m-2},x_{m-1}) &\in \rho_{m-2}(\lambda_{m-2},i_{m-1};\mu_{m-2},j_{m-1}),\\
      (a_{m-1}^{-1}x_{m-1}^{-1}b_{m-1},b_ma_m^{-1}) &\in \rho_{m-1}(\lambda_{m-1},i_m;\mu_{m-1},j_m),
\end{align*}
where for all $1\leq s\leq m-1$, $\rho_s(\lambda_s,i_{s+1};\mu_s,j_{s+1})$ is the rational subset of $G_s\times G_{s+1}^\partial$
consisting of labels of all walks from $(\lambda_s,i_{s+1})$ to $(\mu_s,j_{s+1})$ in $\mathcal{A}(D_s,D_{s+1})$.
\end{thm}

Our immediate goal in the following section is to learn more about these rational subsets $\rho_s$; as it turns out, they actually 
have a faily simple structure. This will allow us to rephrase the theorem just stated in an even more group-theoretical fashion and
hence, to extract decidability results for the word problem of $\ig{\cE}$ significantly more general than Theorem 6.1 in \cite{YDG2}.


\section{Gain graphs and the word problem for $\ig{\cE}$}\label{sec:gain}

Looking at the definition of the contact automaton of two regular $\D$-classes of $\ig{\cE}$, it is not difficult to recognise a particular
instance of a more familiar structure arising in combinatorics (specifically, algebraic graph theory). Namely, let $G$ be a group. We define
a \emph{gain graph} \cite{Gro,MRS,Za} (sometimes also called \emph{voltage graph} or \emph{group labelled graph} \cite[Section 4.13]{RSBook}, 
see also \cite{Ron}) $\Gamma=(V_\Gamma,E_\Gamma,\phi)$ where $(V_\Gamma,E_\Gamma)$ is an undirected graph (with loops and multiple edges allowed) 
and $\phi$ is a labelling (called the \emph{gain function}) of edges $E_\Gamma$ by elements of the group $G$ in a `oriented fashion', so that 
we may conceptualise it as a partial function $\phi:V_\Gamma\times E_\Gamma \times V_\Gamma\to G$ whose domain consists precisely of those triples 
$(u,e,v)$ such that the edge $e$ connects vertices $u$ and $v$. In this sense, we require the condition 
$$
(v,e,u)\phi = [(u,e,v)\phi]^{-1}.
$$
(This will require loops, if any, to be labelled by involutions, elements of $g\in G$ satisfying $g^2=1$, but for most applications loops are
labelled by the identity element.) One of the most important examples of gain graphs in semigroup theory are the so-called \emph{Graham-Houghton
graphs} of a $\D$-class, see, for example, \cite{EG}.

For a walk
$$
\mathbf{w} = u_0e_1u_1\dots u_{n-1}e_nu_n
$$
we define its \emph{gain} $\mathbf{w}\phi$ by
$$
\mathbf{w}\phi = (u_0,e_1,u_1)\phi \dots (u_{n-1},e_n,u_n)\phi.
$$
It is immediate to see that if $\mathbf{w}^{-1}$ denotes the reverse walk of $\mathbf{w}$ then its gain satisfies $\mathbf{w}^{-1}\phi = 
[\mathbf{w}\phi]^{-1}$.

\begin{rmk}\label{rmk:contact}
The contact automaton $\mathcal{A}(D_1,D_2)$ defined in the previous section is essentially a gain graph over the group $G_1\times G_2^\partial$;
to see this, it suffices just to throw away the labels of its transitions from the alphabet $E$ (these were useful in establishing that the subsets 
$\rho_s$ appearing in Theorem \ref{thm:ig-wp} are rational).
\end{rmk}

A \emph{cycle} (based at the vertex $u_0$) is a closed walk without repeated vertices, apart from the first/last vertex $u_0$; similarly, a \emph{path}
is a walk without repeated vertices. We define a \emph{conjugated cycle} based at $u_0$ to be a closed walk of the form $\mathbf{p}\mathbf{c}\mathbf{p}^{-1}$
where $\mathbf{p}$ is a path from $u_0$ to some vertex $v$ and $\mathbf{c}$ is a cycle based at $v$ such that $v$ is the only common vertex of $\mathbf{p}$
and $\mathbf{c}$. (We consider cycles as special cases of conjugated cycles where $\mathbf{p}$ is the empty path.)

Let $W_u\subseteq G$ denote the set of gains of all closed walks in $\Gamma$ based at $u$. The following observation is immediate, thus its proof is omitted.
In fact, we remark that all of the following statements concluding with Theorem \ref{thm:gain-gen} below are basically folklore in gain graph theory, but we 
include their elementary proofs for the sake of completeness.

\begin{lem}\label{lem:Wu}
For any gain graph $\Gamma$ and $u\in V_\Gamma$, $W_u$ is a subgroup of $G$.
\end{lem}

We are going to call $W_u$ the \emph{vertex group} of $u$ in $\Gamma$.

\begin{lem}
If the vertices $u,v$ belong to the same connected component of the gain graph $\Gamma$ then the subgroups $W_u$ and $W_v$ are conjugate.
\end{lem}

\begin{proof}
Let $\mathbf{w}$ be an arbitrary walk from $u$ to $v$, and for brevity write $g=\mathbf{w}\phi$. We claim that $W_u = g W_v g^{-1}$.
Indeed, let $\mathbf{z}$ be any closed walk based at $v$. Then $\mathbf{w}\mathbf{z}\mathbf{w}^{-1}$ is a closed walk based at $u$, so
$g(\mathbf{z}\phi)g^{-1} = (\mathbf{w}\mathbf{z}\mathbf{w}^{-1})\phi \in W_u$. This shows that $g W_v g^{-1}\subseteq W_u$. However, by
switching the roles of vertices $u$ and $v$ in the previous argument (considering an arbitrary closed walk $\mathbf{x}$ based at $u$, 
and noting that $\mathbf{w}^{-1}$ is a walk from $v$ to $u$) we obtain $g^{-1} W_u g\subseteq W_v$, which is equivalent to $W_u\subseteq
g W_v g^{-1}$.
\end{proof}

Since the walk $\mathbf{w}$ in the previous proof was arbitrary, we may conclude the following.

\begin{cor}
We have $gW_vg^{-1} = hW_vh^{-1}$ for gains $g,h\in G$ of any two walks in $\Gamma$ from $u$ to $v$.
\end{cor}

An alternative way of justifying the above statement is to notice that $g^{-1}h$ is a gain of a closed walk based at $v$, so $g^{-1}h\in W_v$.

\begin{lem}\label{lem:coset}
Let $u,v$ be two vertices belonging to the same connected component of the gain graph $\Gamma$. If $W(u,v)$ denotes the set of gains of all 
walks from $u$ to $v$ in $\Gamma$ and $g$ is the gain of an arbitrary fixed walk $\mathbf{w}_0$ from $u$ to $v$, then
$$
W(u,v) = W_u g.
$$
\end{lem}

\begin{proof}
Let $\mathbf{w}$ be an arbitrary walk form $u$ to $v$. Then $\mathbf{w}\mathbf{w}_0^{-1}$ is a closed walk based at $u$, so 
$(\mathbf{w}\mathbf{w}_0^{-1})\phi = (\mathbf{w}\phi)g^{-1} \in W_u$, implying $\mathbf{w}\phi \in W_u g$ and so $W(u,v) \subseteq 
W_u g$. Conversely, any element of the coset $W_u g$ is of the form $(\mathbf{w}\phi) g$ for some closed walk $\mathbf{w}$ based at $u$.
However, this is precisely the gain of the walk $\mathbf{w}\mathbf{w}_0$ which begins at $u$ and ends at $v$. Hence, $W_u g\subseteq
W(u,v)$.
\end{proof}

It is at this point that we can say more about the rational subsets $\rho(\lambda,i;\mu,j)$, introduced in the previous section, that 
comprise labels of all walks in the contact automaton $\mathcal{A}(D_1,D_2)$ from $(\lambda,i)$ to $(\mu,j)$, making a crucial appearance 
in Theorem \ref{thm:ig-wp}: either $\rho(\lambda,i;\mu,j) = \es$ or, otherwise, if $(g_1,g_2)\in G_1\times G_2^\partial$ is the label of 
any fixed walk (say, shortest path, for the sake of an example and computational simplicity) from $(\lambda,i)$ to $(\mu,j)$, then
\begin{equation}\label{eq:rho}
\rho(\lambda,i;\mu,j) = W_{(\lambda,i)}(g_1,g_2).
\end{equation}
However, to see that the above equality is indeed a realisation of $\rho(\lambda,i;\mu,j)$ as a rational subset of $G_1\times G_2^\partial$,
the gain group of $\mathcal{A}(D_1,D_2)$, and to make this formula an explicit and computable (from $\cE$) rational expression representing
$\rho(\lambda,i;\mu,j)$ we must learn more about the vertex group $W_{(\lambda,i)}$; in fact, we would like it to be finitely generated,
with its finite generating set effectively computable from $\mathcal{A}(D_1,D_2)$ (and thus from $\cE$). This is exactly what is supplied
by the following result.

\begin{thm}\label{thm:gain-gen}
Let $\Gamma$ be a gain graph over the group $G$. Then for any $u\in V_\Gamma$, the vertex group $W_u$ is generated by gains of all 
conjugated cycles based at $u$. 

Consequently, if the graph $\Gamma$ is finite then $W_u$ is finitely generated, and there exists an algorithm which, given $u\in V_\Gamma$, 
computes a finite generating set of $W_u$ (represented e.g.\ as words in terms of the generators of $G$).
\end{thm}

\begin{proof}
For the first part of the theorem, we need to show that for any closed walk 
$$
\mathbf{w}=u_0e_1u_1\dots u_{n-1}e_nu_0
$$ 
based at $u=u_0$, $\mathbf{w}\phi\in W_u$ can be written as a product of gains of conjugated cycles based at $u$. We prove this by induction on 
the length of $\mathbf{w}$.

First, if there are no repeated vertices in $\mathbf{w}$ (apart from the initial/final vertex $u$) then $\mathbf{w}$ is itself a cycle (and thus 
a trivially conjugated cycle) based at $u$; its gain is then among the conjectured generating set of $W_u$. 

Otherwise, $\mathbf{w}$ contains repeated vertices, so choose the minimum index $j$ such that $u_j$ represents a repetition in $\mathbf{w}$, 
i.e.\ such that $u_j=u_i$ for some $i<j$. Then $\mathbf{p}=u_0e_1u_1\dots e_iu_i$ is a path in $\Gamma$ (from $u$ to $u_i$) and 
$\mathbf{c}=u_ie_{i+1}u_{i+1}\dots e_ju_j$ is a cycle based at $u_i$, whence $\mathbf{p}$ and $\mathbf{c}$ have no vertices in common except
$u_i$. Hence, $\mathbf{p}\mathbf{c}\mathbf{p}^{-1}$ is a conjugated cycle based at $u$, whereas $\mathbf{p}\mathbf{c}$ is a prefix of $\mathbf{w}$.
Therefore,
\begin{align*}
\mathbf{w}\phi &= (\mathbf{p}\mathbf{c}e_{j+1}u_{j+1}\dots e_nu_0)\phi\\
               &= (\mathbf{p}\phi)(\mathbf{c}\phi)((u_je_{j+1}u_{j+1}\dots e_nu_0)\phi) \\
							 &= (\mathbf{p}\phi)(\mathbf{c}\phi)(\mathbf{p}\phi)^{-1}(\mathbf{p}\phi)((u_je_{j+1}u_{j+1}\dots e_nu_0)\phi) \\
							 &= (\mathbf{p}\mathbf{c}\mathbf{p}^{-1})\phi(\mathbf{p}e_{j+1}u_{j+1}\dots e_nu_0)\phi.
\end{align*}
Now $\mathbf{p}e_{j+1}u_{j+1}\dots e_nu_0 = u_0e_1u_1\dots e_iu_ie_{j+1}u_{j+1}\dots e_nu_0$ is a closed walk based at $u$ whose length is
\emph{strictly shorter} than $\mathbf{w}$; by the induction hypothesis, its gain can be written as a product of gains of some conjugated
cycles based at $u$. We now obtain that the same holds for $\mathbf{w}\phi$, and the induction is thus completed.

For the second part of the statement, note that a finite graph can have only finitely many conjugated cycles based at any of its vertices
(because there are only finitely many paths starting from a given vertex, as well as only finitely many cycles based at a given vertex).
Thus $W_u$ is finitely generated in this case; furthermore, standard search algorithms on graphs suffice to effectively discover all the 
conjugated cycles based at $u$. Upon recording their gains, we recover a generating set of $W_u$ in an algorithmic fashion.
\end{proof}

\begin{rmk}
Now we know that for each vertex $(\lambda,i)$ in the contact automaton $\mathcal{A}(D_1,D_2)$ the vertex group $W_{(\lambda,i)}$ is
finitely generated, and its generating set can be effectively computed from $\mathcal{A}(D_1,D_2)$ (and thus from $\cE$). Also, we
can determine a coset representative $(g_1,g_2)$ as above by identifying a single walk/path from $(\lambda,i)$ to $(\mu,j)$ (if any,
otherwise $\rho(\lambda,i;\mu,j)=\es$) and then reading off its label. This provides a simple effective way to compute the rational
subsets $\rho_s$ appearing in Theorem \ref{thm:ig-wp}.
\end{rmk}

Inspired by the characterisation of the word problem of $\ig{\cE}$ provided by this theorem, let $\mathbf{u}=(i_1,a_1,\lambda_1)
\dots (i_m,a_m,\lambda_m)$ and $\mathbf{v}=(j_1,b_1,\mu_1)\dots(j_m,b_m,\mu_m)$ be two elements with $\D$-fingerprint 
$(D_1,\dots,D_m)$. We define a function $(\cdot,\mathbf{u},\mathbf{v})\theta$ mapping subsets $A\subseteq G_1$ to subsets $G_m$ in 
the manner as described in the following. We construct a sequence $A_k\subseteq G_k$, $1\leq k\leq m$, starting from $A_1=A$. First 
of all, let 
$$
B_{1,2} = (A\times G_2)\cap W_{(\lambda_1,i_2)}g_{1,2}
$$
where $g_{1,2}$ is the label of some fixed walk from $(\lambda_1,i_2)$ to $(\mu_1,j_2)$ in $\mathcal{A}(D_1,D_2)$ (if there is such a walk
at all, otherwise we have just $B_{1,2}=\es$). Then we set $A_2=B_{1,2}\pi_2$ where $\pi_2$ is the second projection map. Assuming that 
$A_k$ has been constructed for some $1<k<m$, let
$$
B_{k,k+1} = (a_k^{-1}A_k^{-1}b_k \times G_{k+1}) \cap W_{(\lambda_k,i_{k+1})}g_{k,k+1}
$$
where $g_{k,k+1}$ is the label of some fixed walk from $(\lambda_k,i_{k+1})$ to $(\mu_k,j_{k+1})$ in the automaton/graph 
$\mathcal{A}(D_k,D_{k+1})$ (again, if any, otherwise $B_{k,k+1}=\es$). Similarly as above, we define $A_{k+1}=B_{k,k+1}\pi_2$. 
It is apparent from this definition that if $A_k=\es$ for some $k$ then $A_p=\es$ for all $p\geq k$. Finally, the value of 
$(A,\mathbf{u},\mathbf{v})\theta$ is the set $A_m\subseteq G_m$ obtained at the end of this iterative process.

Our main interest in the remainder of this paper will be in the case when the set $A$ is a right coset of some finitely generated 
subgroup $H\leq G_1$. To state our main result related to the values and computation of sets of the form 
$(Hx,\mathbf{u},\mathbf{v})\theta$, we need to recall several properties of groups related to finite generation of their subgroups,
and also some algorithmic aspects of these properties. We refer to \cite{DVZ} and the literature cited therein for a more extensive
overview of related concepts.

First, let us recall that a group $G$ has the \emph{Howson property} \cite{Howson} if the intersection of any two finitely subgroups
of $G$ is also finitely generated. For example, free groups, and more generally virtually free groups have this property. We relax
this property by relativising it to a fixed finitely generated subgroup $H\leq G$ and a family $\mathcal{K}$ of finitely generated 
subgroups of $G$: we say that $G$ enjoys the \emph{$(H,\mathcal{K})$-relative Howson property} if for any finitely generated $K\leq G$ 
such that $K\in\mathcal{K}$ we have that $H\cap K$ is finitely generated. An effective version of this relative Howson property is 
denoted $\mathsf{eRHP}(G,H,\mathcal{K})$: it asserts that $G$ has the $(H,\mathcal{K})$-relative Howson property, and furthermore,
that there exists an algorithm which, presented with a finite generating set of $K\leq G$, $K\in\mathcal{K}$, computes a finite 
generating set of $H\cap K$.

A similar relativisation can be done to the well-known \emph{coset intersection problem} $\mathsf{CIP}(G)$ in the group $G$: this problem
requires, given two finitely generated subgroups $H,K\leq G$ (whose generating sets are represented by finite sets of words over the 
generating set of $G$) and $x,y\in G$ (again, represented as words over the generators of $G$), to decide if $Hx\cap Ky=\es$, and in 
the case of a negative answer to compute a right coset representative (as it is easy to show that if $Hx\cap Ky\neq\es$ then $Hx\cap Ky$ 
must be a right coset of $H\cap K$). Now we can fix one of the finitely generated subgroups (say, $H$) and restrict the range of the other 
one ($K$) to a class of finitely generated subgroups $\mathcal{K}$ of $G$, and ask the same question, the \emph{$(H,\mathcal{K})$-relative 
coset intersection problem} $\mathsf{RCIP}(G,H,\mathcal{K})$, where the input consists of a finitely generated $K\leq G$ such that 
$K\in\mathcal{K}$ and $x,y\in G$.

\begin{thm}\label{thm:theta}
Let $\cE$ be a finite biordered set, and let $\mathbf{u},\mathbf{v}$ be two elements of $\ig{\cE}$ of $\D$-fingerprint $(D_1,\dots,D_m)$. 
Let $L\leq G_1$ and $x\in G_1$. Furthermore, let $\mathcal{T}_k$ be the class of \emph{tall} subgroups of $G_k\times G_{k+1}^\partial$, 
$1\leq k<m$, namely subgroups of the form $H\times G_{k+1}$ for a finitely generated $H\leq G_k$.
\begin{itemize}
\item[(1)] $(Lx,\mathbf{u},\mathbf{v})\theta$ is either empty, or a left coset of a subgroup $M$ of $G_m$.
\item[(2)] If for all $1\leq k<m$ the group $G_k\times G_{k+1}^\partial$ has the $(W_{(\lambda_k,i_{k+1})},\mathcal{T}_k)$-relative Howson 
           property, and $L$ is finitely generated, then either $(Lx,\mathbf{u},\mathbf{v})\theta=\es$, or $M$ (as defined in (1)) is 
					 finitely generated as well.
\item[(3)] If for all $1\leq k<m$ the property $\mathsf{eRHP}(G_k\times G_{k+1}^\partial, W_{(\lambda_k,i_{k+1})},\mathcal{T}_k)$ holds and 
           there is an algorithm solving $\mathsf{RCIP}(G_k\times G_{k+1}^\partial, W_{(\lambda_k,i_{k+1})},\mathcal{T}_k)$, then 
					 $(Lx,\mathbf{u},\mathbf{v})\theta$ is effectively computable, i.e., there is an algorithm which, given a finite generating set of 
					 $L$ and $x\in G_1$, decides if $(Lx,\mathbf{u},\mathbf{v})\theta=\es$, and in the case of a negative answer outputs a finite 
					 generating set for $M$ and $y\in G_m$ such that $(Lx,\mathbf{u},\mathbf{v})\theta=yM$.
\end{itemize}
\end{thm}

\begin{proof}
Throughout the proof, we let $\mathbf{u}=(i_1,a_1,\lambda_1)\dots (i_m,a_m,\lambda_m)$ and $\mathbf{v}=(j_1,b_1,\mu_1)\dots(j_m,b_m,\mu_m)$,
with both decompositions arising from minimal r-fac\-to\-ri\-sa\-tions of words representing $\mathbf{u},\mathbf{v}$, respectively.

(1) Let $A_1=Lx,A_2,\dots,A_{m-1},A_m=(\mathbf{u},\mathbf{v},Lx)\theta$ be the sequence of sets constructed in the course of computing
$(\mathbf{u},\mathbf{v},Lx)\theta$. Then, first of all,
$$
B_{1,2} = (Lx \times G_2) \cap W_{(\lambda_1,i_2)}g_{1,2} = (L\times G_2^\partial)(x,1)\cap W_{(\lambda_1,i_2)}(g_{1,2}^{(1)},g_{1,2}^{(2)}),
$$
where $g_{1,2}= (g_{1,2}^{(1)},g_{1,2}^{(2)})$ is the label of some walk $(\lambda_1,i_2)\leadsto(\mu_1,j_2)$ in $\mathcal{A}(D_1,D_2)$
(provided such a walk exists, otherwise $B_{1,2}=\es$).
Now we see that $B_{1,2}$ is the intersection of two right cosets of subgroups of $G_1\times G_2^\partial$ (namely, $L\times G_2^\partial$ 
and $W_{(\lambda_1,i_2)}$), and therefore either $B_{1,2}=\es$ (whence $A_2=\es$ and in fact $A_k=\es$ for all $k\geq 2$), or $B_{1,2}$ is 
a right coset of $(L\times G_2^\partial) \cap W_{(\lambda_1,i_2)}$. However, then $A_2=B_{1,2}\pi_2$ is a right coset of 
$M_2=((L\times G_2^\partial) \cap W_{(\lambda_1,i_2)})\pi_2$, a subgroup of $G_2^\partial$, the dual group of $G_2$. In $G_2$, however, $A_2$ 
is a left coset (of this same subgroup).

We now claim that each of the subsets $A_2,\dots,A_m$ is a left coset of a subgroup of $G_2,\dots,G_m$, respectively. For the purpose of
an inductive argument, assume that $A_k=x_k M_k$ for some $k\geq 2$, some subgroup $M_k$ of $G_k$ and $x_k\in G_k$. Then $A_k^{-1} = M_kx_k^{-1}$
and so, by taking $g_{k,k+1} = (g_{k,k+1}^{(1)},g_{k,k+1}^{(2)})$ to be the label of some walk $(\lambda_k,i_{k+1})\leadsto(\mu_k,j_{k+1})$ in 
$\mathcal{A}(D_k,D_{k+1})$ (if any, otherwise set $B_{k,k+1}=\es$), we have
\begin{align*}
B_{k,k+1} &= (a_k^{-1}A_k^{-1}b_k \times G_{k+1}) \cap W_{(\lambda_k,i_{k+1})}g_{k,k+1} \\
          &= (a_k^{-1}M_kx_k^{-1}b_k \times G_{k+1}) \cap W_{(\lambda_k,i_{k+1})}(g_{k,k+1}^{(1)},g_{k,k+1}^{(2)}) \\
					&= ((a_k^{-1}M_ka_k)(a_k^{-1}x_k^{-1}b_k) \times G_{k+1}) \cap W_{(\lambda_k,i_{k+1})}(g_{k,k+1}^{(1)},g_{k,k+1}^{(2)}) \\
					&= ((a_k^{-1}M_ka_k) \times G_{k+1}^\partial)(a_k^{-1}x_k^{-1}b_k,1) \cap W_{(\lambda_k,i_{k+1})}(g_{k,k+1}^{(1)},g_{k,k+1}^{(2)}),
\end{align*}
so, just as previously, $B_{k,k+1}$ is either empty (implying that $A_{k+1}\neq\es$) or a right coset of $((a_k^{-1}M_ka_k) \times G_{k+1}^\partial) 
\cap W_{(\lambda_k,i_{k+1})}$. In the latter case, $A_{k+1}$ is a right coset of $M_{k+1} = [((a_k^{-1}M_ka_k) \times G_{k+1}^\partial) 
\cap W_{(\lambda_k,i_{k+1})}]\pi_2$ viewed as a subgroup of $G_{k+1}^\partial$, which means that it is a left coset of $M_{k+1}$ within $G_{k+1}$.
Hence, we inductively conclude that $(\mathbf{u},\mathbf{v},Lx)\theta = A_m$ is a left coset of a subgroup of $G_m$.

(2) This follows by a straightforward inspection of the argument (1) above, as we can now inductively show that all of the subgroups $M_1=L,M_2,\dots,M_m$
of $G_1,\dots,G_m$, respectively, are finitely generated. Indeed, for any $k\geq 2$, if $A_k\neq\es$ then $M_k$ arises a second projection of the intersection 
of a tall subgroup of $G_{k-1}\times G_k^\partial$ (namely, $L\times G_2^\partial$ if $k=2$ and $(a_{k-1}^{-1}M_{k-1}a_{k-1}) \times G_k^\partial$ if $k>2$) and 
$W_{(\lambda_{k-1},i_k)}$. Now the assumption that $M_{k-1}$ is finitely generated and the $(W_{(\lambda_{k-1},i_k)},\mathcal{T}_{k-1})$-relative Howson property 
imply that $M_k$ is finitely generated, too. Therefore, if $(Lx,\mathbf{u},\mathbf{v})\theta\neq\es$ then $M_m$ is a finitely generated subgroup of $G_m$.

(3) Again, a careful inspection of the argument (1) suffices to show that the given conditions ensure that the whole procedure as described in (1) is effectively
computable. Namely, assume that the set $A_k$, $k\geq 1$, has already been computed (via a finite generating set $X_k$ for $M_k$ and coset representative $x_k\in G_k$).
Then, if $F_k$ denotes the finite generating set of $G_k$ (supplied by \cite[Theorem 3.10]{DGR}), then $X_1\times F_2$ is a generating set for $L\times G_2^\partial$
and if $k\geq 2$, $(a_k^{-1}X_ka_k) \times F_{k+1}$ is a generating set for $(a_{k-1}^{-1}M_{k-1}a_{k-1}) \times G_k^\partial$. By $\mathsf{RCIP}(G_k\times 
G_{k+1}^\partial, W_{(\lambda_k,i_{k+1})},\mathcal{T}_k)$, there is an algorithm which decides whether 
$$
((a_k^{-1}M_ka_k) \times G_{k+1}^\partial)(a_k^{-1}x_k^{-1}b_k,1) \cap W_{(\lambda_k,i_{k+1})}(g_{k,k+1}^{(1)},g_{k,k+1}^{(2)})=\es.
$$ 
If the answer is ``yes'' then $A_{k+1}=\es$ and thus our algorithm returns $(\mathbf{u},\mathbf{v},Lx)\theta = \es$. Otherwise, by the property 
$\mathsf{eRHP}(G_k\times G_{k+1}^\partial, W_{(\lambda_k,i_{k+1})},\mathcal{T}_k)$ there is an algorithm computing a finite generating set for the intersection 
of subgroups of $G_k\times G_{k+1}^\partial$ whose second projection is $M_{k+1}$, and thus, by taking the second coordinates from these generating elements, 
one effectively computes a finite generating set for $M_{k+1}$. Employing $\mathsf{RCIP}(G_k\times  G_{k+1}^\partial, W_{(\lambda_k,i_{k+1})},\mathcal{T}_k)$ 
once again, we have an algorithm that computes a coset representative for $B_{k,k+1}$, and the second coordinate for this latter pair is the coset representative 
of $A_{k+1}$. By iterating this process, we eventually compute $(\mathbf{u},\mathbf{v},Lx)\theta$.
\end{proof}

We are finally ready to state the announced rephrasing of the word problem of $\ig{\cE}$ in terms of the machinery just introduced.

\begin{thm}\label{thm:ig-wp-re}
With all the notation as in the previous theorem, $\mathbf{u}=\mathbf{v}$ holds in $\ig{\cE}$ for $\mathbf{u}=(i_1,a_1,\lambda_1)\dots 
(i_m,a_m,\lambda_m)$ and $\mathbf{v}=(j_1,b_1,\mu_1)\dots(j_m,b_m,\mu_m)$ if and only if we have $i_1=j_1$, $\lambda_m=\mu_m$ and 
$$
b_ma_m^{-1} \in (\{a_1^{-1}b_1\},\mathbf{u},\mathbf{v})\theta.
$$
\end{thm}

\begin{proof}
($\Rightarrow$)
Assume that we have $\mathbf{u}=\mathbf{v}$. Then, by Theorem \ref{thm:ig-wp} (while also bearing in mind the remarks following Lemma \ref{lem:coset}),
there exist $x_r\in G_r$, $1<r<m$, such that 
\begin{align*}
      (a_1^{-1}b_1, x_2) &\in W_{(\lambda_1,i_2)}g_{1,2},\\
      (a_2^{-1}x_2^{-1}b_2,x_3) &\in W_{(\lambda_2,i_3)}g_{2,3},\\
			&\ \vdots\\
			(a_{m-2}^{-1}x_{m-2}^{-1}b_{m-2},x_{m-1}) &\in W_{(\lambda_{m-2},i_{m-1})}g_{m-2,m-1},\\
      (a_{m-1}^{-1}x_{m-1}^{-1}b_{m-1},b_ma_m^{-1}) &\in W_{(\lambda_{m-1},i_m)}g_{m-1,m},
\end{align*}
where for each $1\leq k<m$, $g_{k,k+1}\in G_k\times G_{k+1}^\partial$ is a label of some walk $(\lambda_k,i_{k+1})\leadsto(\mu_k,j_{k+1})$ in 
$\mathcal{A}(D_k,D_{k+1})$. By the very definition of the process yielding the set $(\{a_1^{-1}b_1\},\mathbf{u},\mathbf{v})\theta$ (during
which we have computed $A_1=\{a_1^{-1}b_1\},A_2,\dots,A_{m-1}$ and $A_m=(\{a_1^{-1}b_1\},\mathbf{u},\mathbf{v})\theta$), we obtain that any 
solution of the above system of constraints satisfies $x_r\in A_r$ for all $2\leq r\leq m-1$. Indeed, this is clear for $r=2$; upon assuming 
that $x_k\in A_k$ for some $k\geq 2$, from the condition 
$$
(a_k^{-1}x_k^{-1}b_k,x_{k+1}) \in W_{(\lambda_k,i_{k+1})}g_{k,k+1}
$$
and the fact that $(a_k^{-1}x_k^{-1}b_k,x_{k+1})\in (a_k^{-1}A_k^{-1}b_k)\times G_{k+1}$ we conclude that $x_{k+1}\in B_{k,k+1}\pi_2=A_{k+1}$.
Thus the last of the given conditions implies that $b_ma_m^{-1}\in A_m=(\{a_1^{-1}b_1\},\mathbf{u},\mathbf{v})\theta$.

($\Leftarrow$) Conversely, assume that $b_ma_m^{-1} \in (\{a_1^{-1}b_1\},\mathbf{u},\mathbf{v})\theta$ holds. Let 
$$A_1=\{a_1^{-1}b_1\},A_2,\dots,A_{m-1},A_m$$ 
be the sets obtained during the iterative process defining $(\{a_1^{-1}b_1\},\mathbf{u},\mathbf{v})\theta$. First of all, our assumption 
implies that none of these sets is empty. Furthermore, $b_ma_m^{-1} \in A_m$ means that there must be a $y_{m-1}\in G_{m-1}$ such that 
$$
(y_{m-1},b_ma_m^{-1})\in ((a_{m-1}^{-1}A_{m-1}^{-1}b_{m-1})\times G_m^\partial)\cap W_{(\lambda_{m-1},i_m)}g_{m-1,m}
$$
for the label $g_{m-1,m}$ of some walk $(\lambda_{m-1},i_m)\leadsto(\mu_{m-1},j_m)$ in $\mathcal{A}(D_{m-1},D_m)$.
This means that $y_{m-1}=a_{m-1}^{-1}x_{m-1}^{-1}b_{m-1}$ for some $x_{m-1}\in A_{m-1}$; so let us fix one such element $x_{m-1}$. Assuming
that $x_k\in A_k$ has already been picked for some $2<k<m$, the fact that we are considering an element of $A_k$ implies the existence of a
$y_{k-1}\in G_{k-1}$ such that 
$$
(y_{k-1},x_{k-1})\in ((a_{k-1}^{-1}A_{k-1}^{-1}b_{k-1})\times G_k^\partial)\cap W_{(\lambda_{k-1},i_k)}g_{k-1,k}
$$
for the label $g_{k-1,k}$ of some walk $(\lambda_{k-1},i_k)\leadsto(\mu_{k-1},j_k)$ in $\mathcal{A}(D_{k-1},D_k)$.
Again, $y_{k-1}=a_{k-1}^{-1}x_{k-1}^{-1}b_{k-1}$ for some $x_{k-1}\in A_{k-1}$, and upon the choice of $y_{k-1}$ with the above property,
this becomes our pick of $x_{k-1}$. This way, we arrive at an element $x_2\in A_2$; however, then
$$
(a_1^{-1}b_1,x_2)\in W_{(\lambda_1,i_2)}g_{1,2}
$$
for the label $g_{1,2}$ of some walk $(\lambda_1,i_2)\leadsto(\mu_1,j_2)$ in $\mathcal{A}(D_1,D_2)$. Bearing in mind Theorem \ref{thm:ig-wp},
it is at this point that we gathered sufficient information to conclude that $\mathbf{u}=\mathbf{v}$.
\end{proof}

The combined effect of the previous two theorems is as follows; this the principal result of this section. Recall that for a group $G$, the
\emph{subgroup membership problem} for $G$ asks for the existence of an algorithm which, given $g\in G$ and a finitely generated subgroup
$H\leq G$, decides whether $g\in H$.

\begin{thm}\label{thm:ig-wp-dec}
Let $\cE$ be a finite biordered sets such that for any two regular $\D$-classes $D_1,D_2$ of $\ig{\cE}$ (with maximal subgroups $G_1,G_2$, 
respectively), each vertex group $W=W_{(\lambda,i)}$ of $\mathcal{A}(D_1,D_2)$ such that there exists a non-regular element of the form 
$(i',g_1,\lambda)(i,g_2,\lambda')\in\ig{\cE}$ for some $i'\in I_1$, $g_1\in G_1$, $g_2\in G_2$, $\lambda'\in\Lambda_2$, and the class 
$\mathcal{T}$ of all tall subgroups of $G_1\times G_2^\partial$ we have that
\begin{itemize}
\item[(i)] the property $\mathsf{eRHP}(G_1\times G_2^\partial,W,\mathcal{T})$ holds, and
\item[(ii)] there is an algorithm solving the problem $\mathsf{RCIP}(G_1\times G_2^\partial,W,\mathcal{T})$.
\end{itemize}
If, in addition, 
\begin{itemize}
\item[(iii)] the maximal subgroups of all regular $\D$-classes of $\ig{\cE}$ have decidable subgroup membership problems,
\end{itemize} 
then the word problem $\ig{\cE}$ is decidable.
\end{thm}

\begin{proof}
First of all, note that for each element $\mathbf{u}=(i_1,a_1,\lambda_1)\dots (i_m,a_m,\lambda_m)\in\ig{\cE}$ of $\D$-fingerprint $(D_1,\dots,D_m)$
we have that $(i_k,a_k,\lambda_k)(i_{k+1},a_{k+1},\lambda_{k+1})$ is not regular, so the assumptions of the theorem are available for any vertex group
of the form $W=W_{(\lambda_k,i_{k+1})}$.

Now, our aim is tho show that there is an algorithm deciding if $\mathbf{u}=\mathbf{v}$ for any given elements $\mathbf{u}=(i_1,a_1,\lambda_1)\dots 
(i_m,a_m,\lambda_m)$ and $\mathbf{v}=(j_1,b_1,\mu_1)\dots(j_m,b_m,\mu_m)$ of $\ig{\cE}$. Clearly, one can establish whether $i_1=j_1$ and 
$\lambda_m=\mu_m$. By the conditions (i)--(ii) and by virtue of Theorem \ref{thm:theta}(3), there is an algorithm which, presented with $\mathbf{u},\mathbf{v}$, 
computes the set $(\{a_1^{-1}b_1\},\mathbf{u},\mathbf{v})\theta$, which is either the empty set or a (left) coset of a finitely generated subgroup $M$ of $G_m$,
and in the latter case computes a finite generating set of $M$ and a coset representative $y\in G_m$.

By Theorem \ref{thm:ig-wp-re}, $\mathbf{u}=\mathbf{v}$ holds if and only if $b_ma_m^{-1} \in (\{a_1^{-1}b_1\},\mathbf{u},\mathbf{v})\theta$. If 
the computed value of $(\{a_1^{-1}b_1\},\mathbf{u},\mathbf{v})\theta$ is $\es$ then our algorithm returns ``no''; otherwise, $\mathbf{u}=\mathbf{v}$
is equivalent to $b_ma_m^{-1}\in yM$, that is, to $y^{-1}b_ma_m^{-1}\in M$. Since by (iii) $G_m$ is assumed to have decidable subgroup membership
problem, we can invoke the corresponding algorithm on input $y^{-1}b_ma_m^{-1}$ and $M$ to decide the latter containment, whence our algorithm for
deciding the word problem of $\ig{\cE}$ is complete.
\end{proof}

\begin{rmk}\label{rmk:almost-finite}
Notice that the above theorem implies, as a corollary, Theorem 6.1 of \cite{YDG2}. Namely, if both $G_1,G_2$ are finite groups then both 
$\mathsf{eRHP}(G_1\times G_2^\partial,W,\mathcal{T})$ (trivially) holds (as it boils down to the computation of the intersection of two subgroups 
of a finite group), and, similarly, there is a straightforward algorithm for solving $\mathsf{RCIP}(G_1\times G_2^\partial,W,\mathcal{T})$ (again, 
because it entails the computation of the intersection of cosets within a finite group). If, however, only one of the groups $G_1,G_2$ is finite,
while the other is free, then $G_1\times G_2^\partial$ is a virtually free group which has the Howson property \cite{Gr}, and, furthermore,
it is effective \cite{Si,Si20}; also the coset intersection problem is algorithmically solvable in virtually free groups \cite{Gr}. 
Finally, if both $G_1,G_2$ are free, then both $D_1,D_2$ are maximal (in the sense of \cite[Section 6]{YDG2}), but then, as shown in 
\cite[Lemma 6.3]{YDG2}, all the vertices of $\mathcal{A}(D_1,D_2)$ appearing in the course of deciding whether an equality $\mathbf{u}=\mathbf{v}$
holds in $\ig{\cE}$ are either isolated or have only a loop around them labelled by the identity element; thus, all the relevant vertex
groups $W$ are trivial, which makes the property $\mathsf{eRHP}(G_1\times G_2^\partial,W,\mathcal{T})$ trivially true. On the other hand,
the problem $\mathsf{RCIP}(G_1\times G_2^\partial,W,\mathcal{T})$ is easily seen to be equivalent to the subgroup membership problem for 
$G_1\times G_2^\partial$ restricted to its (finitely generated) tall subgroups. This, in turn, is equivalent to the subgroup membership problem
for the free group $G_1$, which is decidable (see \cite{BS,Lo}). For the same reason, since by assumptions of \cite[Theorem 6.1]{YDG2} all maximal 
subgroups of $\ig{\cE}$ are finite or free, they all have decidable subgroup membership problems.
\end{rmk}


\section{The structure of $\ig{\cE}$: Green's relations, \Sch groups, shapes of $\D$-classes}\label{sec:str}

In this section our goal is to gain more insight into the structure of (non-regular) $\D$-classes of $\ig{\cE}$. To achieve this, we set
three intermediate objectives: the first one is to provide a clear characterisation of its Green's relations $\R,\L,\H,\D,\J$ and thus to 
describe the classes $R_\mathbf{x},L_\mathbf{x},H_\mathbf{x},D_\mathbf{x},J_\mathbf{x}$ for its typical element written as a product of 
Rees matrix triples $\mathbf{x}=(i_1,a_1,\lambda_1)\dots(i_m,a_m,\lambda_m)$ corresponding to its $\D$-fingerprint. Following this, we turn 
to the question of determining the \Sch group of the $\H$-class $H_\mathbf{x}$ (contained in a non-regular $\D$-class). As we shall see, 
there is a reasonably strong connection between these \Sch groups and the maximal subgroups (in the regular $\D$-classes) of $\ig{\cE}$; 
also, this will provide a handy enumeration of the elements of the $\H$-class $H_\mathbf{x}$. All of this is then applied towards our 
third objective: to supply the basic parameters of the $\D$-class $D_\mathbf{x}$ by which we mean the enumeration of the sets of its 
$\R$-classes and $\L$-classes.

Our pursue of these objectives is based on the following result.

\begin{pro}\label{pro:H}
Let $w\in E^+$ be a word with a minimal r-decomposition $w=p_1\dots p_m$, and let $x,y\in E^*$ be such that $\ol{xwy}=\ol{w}$. Then 
$\ol{xp_1}\;\H\;\ol{p_1}$ and $\ol{p_my}\;\H\; \ol{p_m}$.
\end{pro}

\begin{proof}
First of all, consider the case $m=1$. Let $e$ be a seed for $p_1$, so that $p_1=p'ep''$; then, by Lemma \ref{lem:conv}, $e$ is also a seed
for $xp_1y$. Now by the same lemma $\ol{p'e}\;\L\;\ol{xp'e}$, thus $\ol{p_1}\;\L\;\ol{xp_1}$ (by multiplying by $p''$from the right) 
and $\ol{p_1y}\;\L\;\ol{xp_1y}=\ol{p_1}$ (by multiplying by $p''y$ from the right). Dually, we get $\ol{p_1y}\;\R\;\ol{p_1}\;\R\;
\ol{xp_1}$. Hence, for the rest of the proof we may safely assume that $m\geq 2$.

Let $1=\alpha_1<\alpha_2<\dots<\alpha_m$ be the coordinates of the factorisation $w=p_1\dots p_m$. Then, if $|x|=\xi$, the factorisation
$xwy=xp_1\dots p_my$ has coordinates $(\xi,\alpha_1',\dots,\alpha_m',\xi+|w|+1)$, where $\alpha_i'=\xi+\alpha_i$ for all $1\leq i\leq m$. Furthermore,
let us fix one seed $e_i$ from each of the factors $p_i$ (in accordance with Theorem \ref{thm:DGR}); assume that these occur in $xwy$ at positions
$\gamma_i$, $1\leq i\leq m$, where $\alpha_i'\leq\gamma_i<\alpha_{i+1}'$.

Now, the word $xwy$ itself has some minimal r-factorisation 
$$
xwy = q_1\dots q_m;
$$
note that it must have precisely $m$ factors by Theorem \ref{thm:Dfing} since $\ol{xwy}=\ol{w}$. By the same theorem we have $\ol{p_i}\;\D\;\ol{q_i}$
for all $1\leq i\leq m$; in addition, $\ol{p_1}\;\R\;\ol{q_1}$ and $\ol{p_m}\;\L\;\ol{q_m}$. Let $1=\beta_1<\beta_2<\dots<\beta_m$ be the coordinates 
of this factorisation displayed above. 

We claim that $\beta_2>\gamma_1$, i.e.\ that $q_1$ cannot end before reaching the seed $e_1$ of $p_1$. Seeking a contradiction, assume otherwise. 
Define a function 
$$\varphi:\{1,\dots,m\}\to\{1,\dots,m\}$$ 
by setting $i\varphi = j$ if and only if $j$ is the largest integer such that $\beta_j\leq\gamma_i$, in other words, the position of the seed $e_i$ 
lies within the factor $q_j$. Our assumption tells us that $1\varphi>1$; furthermore, by definition, $\varphi$ is non-decreasing. Now select the 
smallest value of $k$ such that $k\varphi\leq k$ (such $k$ must exist because $m\varphi\leq m$). Then in fact $k\varphi=k$, indeed, otherwise 
$k\varphi\leq k-1$ implying $(k-1)\varphi\leq k-1$, contradicting the choice of $k$. Moreover, $k>1$ (by the above assumption) and $(k-1)\varphi = k$ 
(because otherwise $(k-1)\varphi\leq k-1$, yielding the same contradiction as before). Therefore, $\beta_k\leq\gamma_{k-1}<\gamma_k<\beta_{k+1}$, 
that is, both seeds $e_{k-1}$ and $e_k$ lie within $q_k$. This means that if $p_{k-1}=pe_{k-1}q'$ and $p_k=p'e_kq$ we can write
$$
q_k = ue_{k-1}q'p'e_kv
$$
where either $u$ is a suffix of $p$ or $u=u'p$ for some suffix $u'$ of $p_{k-2}$ not containing its designated seed (or of $x$ if $k=2$), 
and $v$ is either a prefix of $q$ or $v=qv'$ for some prefix $v'$ of $p_{k+1}$ not containing its designated seed (or of $y$ if $k=m$).
If $u$ is a suffix of $p$ then by by Lemma \ref{lem:conv} we have $\ol{ue_{k-1}}\;\L\;\ol{pe_{k-1}}$, and since $\L$ is a right congruence this implies
(upon multiplying by $q'p'e_kv$ from the right)
$$
\ol{q_k}\;\L\; \ol{p_{k-1}p'e_kv}.
$$
On the other hand, if $u=u'p$ with $u'$ as above then, since $\ol{q_k}\;\D\;\ol{p_k}\;\D\;\ol{e_k}$, Lemma \ref{lem:conv} tells us that 
$$\ol{ue_{k-1}q'p'e_k}=\ol{u'pe_{k-1}q'p'e_k}\;\L\;\ol{pe_{k-1}q'p'e_k}=\ol{p_{k-1}p'e_k},$$
which, multiplying by $v$ from the right yields $\ol{q_k}\;\L\; \ol{p_{k-1}p'e_kv}$. Hence, this last relation holds in any case. However, dual arguments
to the ones just presented (which are left to the reader as exercise) would show that 
$$
\ol{p_{k-1}p'e_kv}\;\R\; \ol{p_{k-1}p'e_kq}=\ol{p_{k-1}p_k}.
$$
We conclude that $\ol{p_k}\;\D\;\ol{q_k}\;\D\; \ol{p_{k-1}p_k}$, contradicting the non-regularity of $\ol{p_{k-1}p_k}$.

So, now we know that indeed $\beta_2>\gamma_1$. Dually, we can conclude that $\beta_m\leq\gamma_m$ (meaning that the factor $q_m$ contains the designated
seed $e_m$ of $p_m$; in fact, now it would be rather easy to prove that the function $\varphi$ must be the identity mapping). So, if $p_1=
p_1'e_1p_1''$, then $q_1=xp_1'e_1u$ for some word $u$ such that one of $u,p_1''$ is the prefix of another. Then, by using the seed $e_1$ either in $p_1$
or in $q_1$ and Lemma \ref{lem:conv} as above, we would arrive at the conclusion that $\ol{e_1u}\;\R\;\ol{e_1p_1''}$ and so, by Theorem \ref{thm:Dfing} we
have $\ol{p_1}\;\R\;\ol{q_1}=\ol{xp_1'e_1u}\;\R\;\ol{xp_1'e_1p_1''}=\ol{xp_1}$. In particular, $\ol{p_1}\;\D\;\ol{xp_1}$. But then \cite[Proposition 4.1(2)]{YDG2}
implies that we also have $\ol{p_1}\;\L\;\ol{xp_1}$, thus $\ol{p_1}\;\H\;\ol{xp_1}$. The relation $\ol{p_my}\;\H\; \ol{p_m}$ follows by left-right
duality.
\end{proof}

\begin{thm}\label{thm:Green}
Let $\cE$ be a finite biorder, and let $\mathbf{x}=(i_1,a_1,\lambda_1)\dots(i_m,a_m,\lambda_m)\in\ig{\cE}$ with $\D$-fingerprint $(D_1,\dots,D_m)$. 
Furthermore, let $\mathbf{y}\in\ig{\cE}$.
\begin{itemize}
\item[(1)] $\mathbf{x}\;\R\;\mathbf{y}$ if and only if $\mathbf{y}=(i_1,a_1,\lambda_1)\dots(i_m,b,\mu)$ for some $b\in G_m$ and $\mu\in\Lambda_m$.
\item[(2)] $\mathbf{x}\;\L\;\mathbf{y}$ if and only if $\mathbf{y}=(j,c,\lambda_1)\dots(i_m,a_m,\lambda_m)$ for some $c\in G_1$ and $j\in I_1$.
\item[(3)] $\mathbf{x}\;\H\;\mathbf{y}$ if and only if there exist $b\in G_m$, $c\in G_1$, $\mu\in\Lambda_m$ and $j\in I_1$ such that
$$
\mathbf{y} = (i_1,a_1,\lambda_1)\dots(i_m,b,\mu) = (j,c,\lambda_1)\dots(i_m,g,\lambda_m)
$$
holds in $\ig{\cE}$.
\item[(4)] $\mathbf{x}\;\D\;\mathbf{y}$ if and only if there exist $b\in G_m$, $c\in G_1$, $\mu\in\Lambda_m$ and $j\in I_1$ such that
$$
\mathbf{y} = (j,c,\lambda_1)\dots(i_m,b,\mu)
$$
holds in $\ig{\cE}$.
\item[(5)] $\D=\J$ holds in $\ig{\cE}$.
\end{itemize}
\end{thm}

\begin{proof}
Throughout the proof, we assume that $u,v\in E^*$ are such that $\mathbf{x}=\ol{u}$ and $\mathbf{y}=\ol{v}$, where the given decomposition of $\mathbf{x}$
is induced by a minimal r-factorisation of $u$, namely $u=p_1\dots p_m$ (so that $\ol{p_k}=(i_k,a_k,\lambda_k)$).

(1) The statement $\mathbf{x}\;\R\;\mathbf{y}$ is equivalent to the existence of words $s,t\in E^*$ such that $\ol{us}=\ol{v}$ and $\ol{vt}=\ol{u}$. In
particular, $\ol{ust}=\ol{u}$, so by the previous proposition $\ol{p_mst}\;\H\;\ol{p_m}$. This implies $\ol{p_ms}\;\R\;\ol{p_m}$, prompting $v$ to have 
the same $\D$-fingerprint as $u$. Hence, $\ol{p_ms}=(i_m,b,\mu)$ for some $b\in G_m$ and $\mu\in\Lambda_m$, implying $\mathbf{y}=\ol{v}=\ol{us}=
\ol{p_1\dots p_ms}$ to have the indicated form. Conversely, as $(i_m,a_m,\lambda_m)\;\R\;(i_m,b,\mu)$ for any $b\in G_m$ and $\mu\in\Lambda_m$ and 
$\R$ is a left congruence, any element $\mathbf{y}$ of the given form is $\R$-related to $\mathbf{x}$.

(2) This follows by an argument that is completely dual to (1).

(3) This is an immediate consequence of (1) and (2).

(4) We have $\mathbf{x}\;\D\;\mathbf{y}$ if and only if $\mathbf{x}\;\R\;\mathbf{z}\;\L\;\mathbf{y}$ for some $\mathbf{z}\in\ig{\cE}$. From the first 
relation it follows that any such $\mathbf{z}$ must have the form $(i_1,a_1,\lambda_1)\dots(i_m,b,\mu)$ for some $b\in G_m$ and $\mu\in\Lambda_m$,
whence $\mathbf{z}\;\L\;\mathbf{y}$ implies that $\mathbf{y}$ has the required form. Conversely, if $\mathbf{y} = (j,c,\lambda_1)\dots(i_m,b,\mu)$
for arbitrary $b\in G_m$, $c\in G_1$, $\mu\in\Lambda_m$ and $j\in I_1$, then $\mathbf{x}\;\R\;\mathbf{z}\;\L\;\mathbf{y}$ for $\mathbf{z}=
(i_1,a_1,\lambda_1)\dots(i_m,b,\mu)$.

(5) Assume that $\mathbf{x}\;\J\;\mathbf{y}$; then there are $q,r,s,t\in E^*$ such that $\ol{qur}=\ol{v}$ and $\ol{svt}=\ol{u}$. Therefore, $\ol{squrt}=\ol{u}$,
so the previous proposition tells us that $\ol{sqp_1}\;\H\;\ol{p_1}$ and $\ol{p_mrt}\;\H\;\ol{p_m}$. In particular, $\ol{p_mrt}\;\R\;\ol{p_m}$, so (upon
multiplication by $\ol{p_1\dots p_{m-1}}$ from the left) we have $\ol{urt}\;\R\;\ol{u}$. Similarly, $\ol{sqp_1}\;\L\;\ol{p_1}$ implies $\ol{squrt}\;\L\;\ol{urt}$.
Thus we have $\mathbf{x}\;\R\;\ol{urt}\;\L\;\mathbf{y}$, i.e.\ $\mathbf{x}\;\D\;\mathbf{y}$.
\end{proof}

Bearing in mind Theorem \ref{thm:ig-wp-re} and using the map $\theta$ introduced in the previous section, there is an even more compact way
of expressing Green's relations in $\ig{\cE}$.

\begin{cor}\label{cor:Green}
Let $$\mathbf{x}=(i_1,a_1,\lambda_1)\dots(i_m,a_m,\lambda_m)$$ and $$\mathbf{y}=(j_1,b_1,\mu_1)\dots(j_m,b_m,\mu_m)$$ be two elements of $\ig{\cE}$
of $\D$-fingerprint $(D_1,\dots,D_m)$. 
\begin{itemize}
\item[(1)] $\mathbf{x}\;\R\;\mathbf{y}$ if and only if $i_1=j_1$ and $(\{a_1^{-1}b_1\},\mathbf{x},\mathbf{y})\theta\neq\es$.
\item[(2)] $\mathbf{x}\;\L\;\mathbf{y}$ if and only if $\lambda_m=\mu_m$ and $b_ma_m^{-1}\in(G_1,\mathbf{x},\mathbf{y})\theta$.
\item[(3)] $\mathbf{x}\;\H\;\mathbf{y}$ if and only if $i_1=j_1$, $\lambda_m=\mu_m$, $(\{a_1^{-1}b_1\},\mathbf{x},\mathbf{y})\theta\neq\es$ and
$b_ma_m^{-1}\in(G_1,\mathbf{x},\mathbf{y})\theta$.
\item[(4)] $\mathbf{x}\;\D\;\mathbf{y}$ if and only if $(G_1,\mathbf{x},\mathbf{y})\theta\neq\es$.
\end{itemize}
\end{cor}

\begin{proof}
(1) By the previous theorem, we have $\mathbf{x}\;\R\;\mathbf{y}$ if and only if there exist $b\in G_m$ and $\mu\in\Lambda_m$ such that 
$$
(i_1,a_1,\lambda_1)\dots(i_m,b,\mu)=(j_1,b_1,\mu_1)\dots(j_m,b_m,\mu_m).
$$
By Theorem \ref{thm:ig-wp-re}, this is equivalent to $i_1=j_1$, $\mu=\mu_m$ (the sole effect of this being that $\mu$ is uniquely determined), and
$b_mb^{-1}\in(\{a_1^{-1}b_1\},\mathbf{x},\mathbf{y})\theta$. So, the condition $\mathbf{x}\;\R\;\mathbf{y}$ is equivalent to the existence of
$b\in G_m$ such that $b\in [(\{a_1^{-1}b_1\},\mathbf{x},\mathbf{y})\theta]b_m$, which is in turn equivalent to $(\{a_1^{-1}b_1\},\mathbf{x},\mathbf{y})\theta$
being non-empty.

(2) We have $\mathbf{x}\;\L\;\mathbf{y}$ if and only if 
$$
(j,c,\lambda_1)\dots(i_m,a_m,\lambda_m)=(j_1,b_1,\mu_1)\dots(j_m,b_m,\mu_m)
$$
for some $c\in G_1$ and $j\in I_1$. Again by Theorem \ref{thm:ig-wp-re}, this is equivalent to $j=j_1$, $\lambda_m=\mu_m$ and the existence of $c\in G_1$
such that $b_ma_m^{-1} \in (\{c^{-1}b_1\},\mathbf{x},\mathbf{y})\theta$. Clearly, the latter condition is equivalent to saying that $b_ma_m^{-1} \in 
(\{g\},\mathbf{x},\mathbf{y})\theta$ for some $g\in G_1$, i.e.\ that $b_ma_m^{-1} \in (G_1,\mathbf{x},\mathbf{y})\theta$.

(3) This follows directly from (1) and (2).

(4) The condition $\mathbf{x}\;\D\;\mathbf{y}$ is equivalent to the equality
$$
(j,c,\lambda_1)\dots(i_m,b,\mu) = (j_1,b_1,\mu_1)\dots(j_m,b_m,\mu_m)
$$
holding for some $b\in G_m$, $c\in G_1$, $\mu\in\Lambda_m$ and $j\in I_1$. Employing Theorem \ref{thm:ig-wp-re} yet again, the latter assertion is further
equivalent to $j=j_1$, $\mu=\mu_m$, and the condition $b_mb^{-1}\in(\{c^{-1}b_1\},\mathbf{x},\mathbf{y})\theta$ holding for some $b\in G_m$ and $c\in G_1$. 
Similarly as in (1) and (2), this is the same as saying that $(\{g\},\mathbf{x},\mathbf{y})\theta$ is not empty for some $g\in G_1$, i.e.\ that 
$(G_1,\mathbf{x},\mathbf{y})\theta\neq\es$.
\end{proof}

An immediate application of the previous result concerns decidability of Green's relations: given two elements of $\ig{\cE}$ we would like
to decide if they are $\R$-/$\L$-/$\H$-/$\D$-related. We instantly see that under the same assumptions as in Theorem \ref{thm:ig-wp-dec} (which 
guarantee decidability of the word problem) such algorithms exist, because basically they ensure that the map $\theta$ is effectively computable 
in the sense discussed in Theorem \ref{thm:theta}(3) and that one can decide membership in any given finitely generated subgroup of a maximal 
subgroup of $\ig{\cE}$. 

\begin{thm}
Let $\cE$ be a finite biordered set satisfying the conditions (i)--(iii) of Theorem \ref{thm:ig-wp-dec}. Then $\ig{\cE}$ has decidable Green's
relations.
\end{thm}

\begin{rmk}
As noticed in Remark \ref{rmk:almost-finite}, the previous result includes finite biorders satisfying the assumptions of \cite[Theorem 6.1]{YDG2}
(described in detail on p.1029-1030 of that paper), where all non-maximal regular $\D$-classes of $\ig{\cE}$ have finite maximal subgroups.
\end{rmk}

We now move on to our second objective, which revolves around the question of computing the \Sch group associated with the $\D$-class of an element
$\mathbf{x}\in\ig{\cE}$. This group determines a great deal the structural features of $D_\mathbf{x}$. \Sch groups were originally introduced by
\Sch \cite{Sch1,Sch2} as generalisations of maximal subgroups of regular $\D$-classes to arbitrary $\D$-classes.

So, let $D$ be an arbitrary $\D$-class of a semigroup $S$, and let $H\subseteq D$ be one of its $\H$-classes. The set of all elements 
$s\in S^1$ such that $Hs\subseteq H$ is denoted by $\mathrm{Stab}(H)$ and called the (\emph{right}) \emph{stabiliser} of $H$. Then Green's Lemma 
\cite[Lemma 2.2.1]{HoBook} ensures that $s\in\mathrm{Stab}(H)$ implies the stronger assertion $Hs=H$ and that the right translation mapping
$\rho_s:h\mapsto hs$, $h\in H$, is a permutation of $H$. Now define an equivalence $\sigma_H$ on $\mathrm{Stab}(H)$ by $(s,t)\in\sigma_H$ if 
and only if $\rho_s=\rho_t$ (i.e.\ if $s,t$ induce the same permutation on $H$). Then the quotient set $\mathrm{Stab}(H)/\sigma_H$ (actually,
$\mathrm{Stab}(H)$ is a submonoid of $S^1$ and $\sigma_H$ is readily seen to be a congruence, so the latter quotient is a well-defined monoid)
is naturally identified with the collection of permutations $\{\rho_s: s\in\mathrm{Stab}(H)\}$, and the latter is easily seen to be a group: 
this is the \emph{(right) \Sch group} $\Gamma_H$ of $H$. Here are several well-known basic facts about \Sch groups (see \cite{HoBook}):
\begin{itemize}
\item the permutation group $\Gamma_H$ acts regularly on $H$, and thus $|\Gamma_H|=|H|$;
\item if $H$ is itself a group (implying that $D$ is regular) then $\Gamma_H\cong H$;
\item if $H'$ is another $\H$-class contained in $D$ then $\Gamma_{H'}\cong\Gamma_H$.
\end{itemize}
Because of this latter property, the \Sch group is indeed an invariant of a $\D$-class. In a dual manner, one can define the \emph{left \Sch group}
$\ol\Gamma_H$ of an $\H$-class $H$; however, it is known that we always have $\ol\Gamma_H\cong\Gamma_H$, so in this sense we may simply speak about
the \Sch group of an $\H$-class (and so in fact of a $\D$-class).

We begin with a number of quite relevant observations.

\begin{pro}\label{pro:theta-subg}
For all finite biordered sets $\cE$, the following holds in $\ig{\cE}$ for an element $\mathbf{x}=(i_1,a_1,\lambda_1)\dots(i_m,a_m,\lambda_m)$
of $\D$-fingerprint $(D_1,\dots,D_m)$:
\begin{itemize}
\item[(1)] For any subgroup $K\leq G_1$, $(K,\mathbf{x},\mathbf{x})\theta$ is a subgroup of $G_m$.
\item[(2)] For any $g\in G_m$ we have $g\in (E,\mathbf{x},\mathbf{x})\theta$ if and only if there exist $x_r\in G_r$, $2\leq r\leq m-1$, such that 
\begin{align*}
      (1, x_2) &\in W_{(\lambda_1,i_2)},\\
      (a_2^{-1}x_2^{-1}a_2,x_3) &\in W_{(\lambda_2,i_3)},\\
			&\ \vdots\\
			(a_{m-2}^{-1}x_{m-2}^{-1}a_{m-2},x_{m-1}) &\in W_{(\lambda_{m-2},i_{m-1})},\\
      (a_{m-1}^{-1}x_{m-1}^{-1}a_{m-1},g) &\in W_{(\lambda_{m-1},i_m)}.
\end{align*}
\item[(3)] For any $g\in G_m$ we have $g\in (G_1,\mathbf{x},\mathbf{x})\theta$ if and only if there exist $y_r\in G_r$, $1\leq r\leq m-1$, such that 
\begin{align*}
      (y_1, y_2) &\in W_{(\lambda_1,i_2)},\\
      (a_2^{-1}y_2^{-1}a_2,y_3) &\in W_{(\lambda_2,i_3)},\\
			&\ \vdots\\
			(a_{m-2}^{-1}y_{m-2}^{-1}a_{m-2},y_{m-1}) &\in W_{(\lambda_{m-2},i_{m-1})},\\
      (a_{m-1}^{-1}y_{m-1}^{-1}a_{m-1},g) &\in W_{(\lambda_{m-1},i_m)}.
\end{align*}
\item[(4)] $(E,\mathbf{x},\mathbf{x})\theta \lnor (G_1,\mathbf{x},\mathbf{x})\theta$.
\end{itemize}
\end{pro}

\begin{proof}
(1) For this, it suffices to analyse the procedure for computing $(\cdot,\mathbf{x},\mathbf{x})\theta$ when the argument is a subgroup (and when 
the two parameters $\mathbf{u},\mathbf{v}$ are equal). Namely, in this case we can take $g_{k,k+1}=(1,1)$ (as the label of the empty walk) so that 
every involved coset of a vertex group is in fact the vertex group itself. Further, we have $b_k=a_k$ for all $1\leq k\leq m$, and the consequence of this
is that all sets $B_{k,k+1}$, $1\leq k<m$, are subgroups of $G_k\times G_{k+1}^\partial$ (and thus never empty) provided $A_k$ is a subgroup of $G_k$. 
Consequently, if $A_1=K$ is a subgroup of $G_1$, then each of $A_2\dots,A_m$ is a subgroup of $G_2,\dots,G_m$, respectively. Since 
$(K,\mathbf{x},\mathbf{x})\theta=A_m\leq G_m$, the required conclusion follows.

(2) For $g\in G_m$ define $\mathbf{x}_g=(i_1,a_1,\lambda_1)\dots(i_m,ga_m,\lambda_m)$ and notice that 
$$(A,\mathbf{x},\mathbf{x}_g)\theta = (A,\mathbf{x},\mathbf{x})\theta$$ 
holds for any $A\subseteq G_1$ (because the function $(\cdot,\mathbf{u},\mathbf{v})\theta$ by its definition does not depend on the group components 
of the rightmost factors of their parameters $\mathbf{u},\mathbf{v}$). Hence, by Theorem \ref{thm:ig-wp-re} we have $g=ga_ma_m^{-1} \in 
(E,\mathbf{x},\mathbf{x})\theta = (E,\mathbf{x},\mathbf{x}_g)\theta$ if and only if $\mathbf{x}=\mathbf{x}_g$ holds in $\ig{\cE}$. However, this is 
by Theorem \ref{thm:ig-wp} just equivalent to the displayed system of constraints.

(3) First of all, note that $g\in (G_1,\mathbf{x},\mathbf{x})\theta$ if and only if $g\in (\{y_1\},\mathbf{x},\mathbf{x})\theta$ for some $y_1\in G_1$.
Define $\mathbf{x}_{y_1,g}=(i_1,a_1y_1,\lambda_1)\dots(i_m,ga_m,\lambda_m)$ and notice, similarly as in (2) above, that 
$$(A,\mathbf{x},\mathbf{x}_{y_1,g})\theta=(A,\mathbf{x},\mathbf{x})\theta$$ 
holds for any $A\subseteq G_1$. Therefore, $g\in (\{y_1\},\mathbf{x},\mathbf{x})\theta = (\{y_1\},\mathbf{x},\mathbf{x}_{y_1,g})\theta$ is by 
Theorem \ref{thm:ig-wp-re} equivalent to $\mathbf{x}=\mathbf{x}_{y_1,g}$ holding in $\ig{\cE}$, which is, in turn, equivalent to the given system 
by Theorem \ref{thm:ig-wp}.

(4) From the preceding parts of this proposition it is immediate that $(E,\mathbf{x},\mathbf{x})\theta$ is a subgroup of 
$(G_1,\mathbf{x},\mathbf{x})\theta$. To see that it is a normal subgroup, let $g\in (E,\mathbf{x},\mathbf{x})\theta$ and $h\in 
(G_1,\mathbf{x},\mathbf{x})\theta$. Further, let $x_r\in G_r$, $2\leq r\leq m-1$, and $y_r\in G_r$, $1\leq r\leq m-1$, be solutions of the systems 
given in (2) and (3) witnessing these that $g,h$ belong to their respective subgroups. Then:
$$
(1, y_2*x_2*y_2^{-1}) = (y_1,y_2)(1,x_2)(y_1,y_2)^{-1} \in W_{(\lambda_1,i_2)},
$$
for any $2\leq k\leq m-2$ we have
$$
(a_k^{-1}y_k^{-1}x_k^{-1}y_ka_k,y_{k+1}*x_{k+1}*y_{k+1}^{-1}) = $$
$$ = (a_k^{-1}y_k^{-1}a_k,y_{k+1})(a_k^{-1}x_k^{-1}a_k,x_{k+1})(a_k^{-1}y_k^{-1}a_k,y_{k+1})^{-1} \in W_{(\lambda_k,i_{k+1})},
$$
and, finally,
$$
(a_{m-2}^{-1}y_{m-2}^{-1}x_{m-2}^{-1}y_{m-2}a_{m-2},h*g*h^{-1}) = $$
$$= (a_{m-1}^{-1}y_{m-1}^{-1}a_{m-1},h)(a_{m-1}^{-1}x_{m-1}^{-1}a_{m-1},g)(a_{m-1}^{-1}y_{m-1}^{-1}a_{m-1},h)^{-1} \in W_{(\lambda_{m-1},i_m)}.
$$
However, by (2), the meaning of this system of constraints is precisely that the elements $z_r=y_r*x_r*y_r^{-1}=y_r^{-1}x_ry_r\in G_r$, $2\leq r\leq m-1$,
witness that $g*h*g^{-1}=g^{-1}hg\in G_m$ belongs to $(E,\mathbf{x},\mathbf{x})\theta$.
\end{proof}

We are now ready to state the characterisation of the \Sch group of the $\H$-class containing an element $\mathbf{x}$ as above.

\begin{thm}\label{thm:Sch}
Assuming the notation and terminology from the previous proposition, we have
$$
\Gamma_{H_\mathbf{x}} \cong (G_1,\mathbf{x},\mathbf{x})\theta / (E,\mathbf{x},\mathbf{x})\theta.
$$
\end{thm}

\begin{proof}
We begin with the following key observation.
\begin{cla}
For any $\mathbf{s}\in\mathrm{Stab}(H_\mathbf{x})$ there exists an element $\mathbf{h}\in H_{(i_m,a_m,\lambda_m)}$ (where the latter is a group
$\H$-class contained in $D_m$, isomorphic to $G_m$) such that $\rho_\mathbf{s} = \rho_\mathbf{h}$, i.e.\ which induces the same right translation 
on $H_\mathbf{x}$ as $\mathbf{s}$. 
\end{cla}
\begin{proof}[Proof of Claim]
Assume that $\mathbf{s}=\ol{s}$ for a word $s\in E^*$. If $w$ is a word representing $\mathbf{x}$ (and giving rise to its decomposition as stated in 
Proposition \ref{pro:theta-subg}, induced by a minimal r-factorisation $w=p_1\dots p_m$) then $\ol{ws}\;\H\;\ol{w}$. In particular, $\ol{ws}\;\L\;\ol{w}$; 
consequently, there is a word $x\in E^*$ such that $\ol{xws}=\ol{w}$. By Proposition \ref{pro:H} this implies that $\ol{p_ms}\;\H\;\ol{p_m}$, that is,
\begin{align*}
(i_m,b,\lambda_m)\mathbf{s} &= (i_m,b,\lambda_m)(i_m,a_m,\lambda_m)^{-1}(i_m,a_m,\lambda_m)\mathbf{s}\\
&= (i_m,b,\lambda_m)(i_m,a_m,\lambda_m)^{-1}\mathbf{h}',
\end{align*}
where the inverse is taken within the group $H_{(i_m,a_m,\lambda_m)}$. Upon defining $\mathbf{h}=(i_m,a_m,\lambda_m)^{-1}\mathbf{h}'\in H_{(i_m,a_m,\lambda_m)}$, 
we see that multiplication from the right by $\mathbf{s}$ and $\mathbf{h}$ acts the same on the following subset of $\ig{\cE}$:
$$
H'=\{(i_1,a_1,\lambda_1)\dots(i_m,b,\lambda_m):\ b\in G_m\}.
$$
However, by Theorem \ref{thm:Green}(3) we have that $H_\mathbf{x}\subseteq H'$ (in fact, $H_\mathbf{x}$ is precisely the collection of all elements of $H'$ that
can be written in the form $(i_1,c,\lambda_1)\dots(i_m,a_m,\lambda_m)$ for some $c\in G_1$), so $\rho_\mathbf{s} = \rho_\mathbf{h}$, as wanted.
\end{proof}

Due to this claim, we now know that the (monoid) quotient $\mathrm{Stab}(H_\mathbf{x})/\sigma_{H_\mathbf{x}}$ is isomorphic to the quotient of the subgroup
$G=\mathrm{Stab}(H_\mathbf{x}) \cap H_{(i_m,a_m,\lambda_m)}$ of $H_{(i_m,a_m,\lambda_m)}$ by the corresponding restriction of $\sigma_{H_\mathbf{x}}$, i.e.\
by the normal subgroup $N$ consisting of all elements of $G$ inducing the identity map as their right translations of $H_\mathbf{x}$.

Now, $(i_m,g,\lambda_m)\in G$ if and only if for any 
\begin{equation}\label{eq:H}
(i_1,a_1,\lambda_1)\dots(i_m,b,\lambda_m)=(i_1,c,\lambda_1)\dots(i_m,a_m,\lambda_m)\in H_\mathbf{x}
\end{equation}
we have that 
$$
(i_1,a_1,\lambda_1)\dots(i_m,b,\lambda_m)(i_m,g,\lambda_m)=(i_1,a_1,\lambda_1)\dots(i_m,bp_{\lambda_mi_m}g,\lambda_m)
$$
(where $p_{\lambda_mi_m}$ is the corresponding element of the sandwich matrix for $D_m$) can be also expressed in the form  $(i_1,d,\lambda_1)\dots(i_m,a_m,\lambda_m)$ 
for some $d\in G_1$. By Theorem \ref{thm:ig-wp-re} and remarks analogous to those made in the proofs of (2) and (3) of Proposition \ref{pro:theta-subg}, this is 
the same as saying that 
$$
bp_{\lambda_mi_m}ga_m^{-1} \in (\{d^{-1}a_1\},\mathbf{x},\mathbf{x})\theta
$$
for some $d\in G_1$. Since $d^{-1}a_1$ traverses the entire $G_1$ as $d$ traverses $G_1$, a further equivalent statement is that 
$$
bp_{\lambda_mi_m}ga_m^{-1} \in (G_1,\mathbf{x},\mathbf{x})\theta.
$$
However, from \eqref{eq:H} and Theorem \ref{thm:ig-wp-re} we must have
$ba_m^{-1}\in (\{a_1^{-1}c\},\mathbf{x},\mathbf{x})\theta$ for some $c\in G_1$, that is,
$$
ba_m^{-1}\in (G_1,\mathbf{x},\mathbf{x})\theta.
$$
As $bp_{\lambda_mi_m}ga_m^{-1}=(ba_m^{-1})(a_m(p_{\lambda_mi_m}g)a_m^{-1})$, we deduce that $(i_m,g,\lambda_m)\in G$ if and only if 
$$
p_{\lambda_mi_m}g \in a_m^{-1}[(G_1,\mathbf{x},\mathbf{x})\theta]a_m.
$$
It is a straightforward exercise to show that the map 
$$\psi:(i_m,g,\lambda_m)\mapsto p_{\lambda_mi_m}g$$
is a group isomorphism $H_{(i_m,a_m,\lambda_m)} \to G_m$. Hence, as $g$ traverses $p^{-1}_{\lambda_mi_m}K$, where $K=a_m^{-1}[(G_1,\mathbf{x},\mathbf{x})\theta]a_m$, the images
of $(i_m,g,\lambda_m)$ traverse the subgroup $K\leq G_m$. Thus $G\psi=K$.

In a similar fashion, consider the assertion $(i_m,g,\lambda_m)\in N$. This is true if and only if for any $b\in G_1$ such that $ba_m^{-1}\in (G_1,\mathbf{x},\mathbf{x})\theta$
we have 
$$
(i_1,a_1,\lambda_1)\dots(i_m,bp_{\lambda_mi_m}g,\lambda_m) = (i_1,a_1,\lambda_1)\dots(i_m,b,\lambda_m).
$$
By using Theorem \ref{thm:ig-wp-re} once again, we derive that another equivalent of this assertion is 
$$
bp_{\lambda_mi_m}gb^{-1}\in (E,\mathbf{x},\mathbf{x})\theta.
$$
However, as $bp_{\lambda_mi_m}gb^{-1} = (ba_m^{-1})(a_m(p_{\lambda_mi_m}g)a_m^{-1})(ba_m^{-1})^{-1}$ and  since by Proposition \ref{pro:theta-subg} we have
$(E,\mathbf{x},\mathbf{x})\theta\lnor (G_1,\mathbf{x},\mathbf{x})\theta$, this can be further simplified to
$$
p_{\lambda_mi_m}g \in a_m^{-1}[(E,\mathbf{x},\mathbf{x})\theta]a_m.
$$
By considering the isomorphism $\psi$ once more, we get $N\psi = L = a_m^{-1}[(E,\mathbf{x},\mathbf{x})\theta]a_m$. Finally,
$$
\Gamma_{H_\mathbf{x}} \cong G/N \cong K/L \cong (G_1,\mathbf{x},\mathbf{x})\theta / (E,\mathbf{x},\mathbf{x})\theta, 
$$
as required.
\end{proof}

Recall that it is customary in group theory to call a quotient of a subgroup of a group $G$ a \emph{divisor} of $G$.

\begin{cor}
Any \Sch group of $\ig{\cE}$ is a divisor of some maximal subgroup of $\ig{\cE}$.
\end{cor}

In particular, with the notation as in Theorem \ref{thm:Sch}, $\Gamma_{H_\mathbf{x}}$ is a divisor of $G_m$.

Now we turn towards the computational aspects of the considered \Sch groups. Just as Proposition \ref{pro:theta-subg}(1) above
is the counterpart of Theorem \ref{thm:theta}(1), we need analogues of parts (2) and (3) of that theorem, dealing with the issue
of finite generation and effective computability.

\begin{thm}
Let $\cE$ be a finite biordered set and assume that $\mathbf{x}\in\ig{\cE}$ has $\D$-fingerprint
$(D_1,\dots,D_m)$. Furthermore, let $\mathcal{T}_k$ be the class of \emph{tall} subgroups of $G_k\times G_{k+1}^\partial$, $1\leq k<m$.
\begin{itemize}
\item[(1)] If for all $1\leq k<m$ the group $G_k\times G_{k+1}^\partial$ has the $(W_{(\lambda_k,i_{k+1})},\mathcal{T}_k)$-relative Howson 
           property and $K\leq G_1$ is finitely generated then $(K,\mathbf{x},\mathbf{x})\theta$ is a finitely generated subgroup of $G_m$.
					 In particular, the \Sch group $\Gamma_{H_\mathbf{x}}$ is finitely generated.
\item[(2)] If for all $1\leq k<m$ the property $\mathsf{eRHP}(G_k\times G_{k+1}^\partial, W_{(\lambda_k,i_{k+1})},\mathcal{T}_k)$ holds then 
					 $(K,\mathbf{x},\mathbf{x})\theta$ is effectively computable, i.e., there is an algorithm which, given a finite generating set of 
					 $K$, computes a finite generating set for $L=(K,\mathbf{x},\mathbf{x})\theta$.
\end{itemize}
\end{thm}

\begin{proof}
This follows by \emph{mutatis mutandis} the proofs of the parts (2) and (3) of Theorem \ref{thm:theta}, bearing in mind the simplified process
of computing $(K,\mathbf{x},\mathbf{x})\theta$ for a subgroup $K$ of $G_1$ sketched in Proposition \ref{pro:theta-subg}(1). For (1), it suffices
to see that the given relative Howson properties ensure that if $A_1=K$ is a finitely generated subgroup of $G_1$, so are all of $A_2,\dots,A_m$
(within $G_2,\dots,G_m$, respectively. Indeed, if $A_k\leq G_k$ is finitely generated, then $B_{1,2}=(A_1\times G_2^\partial)\cap W_{(\lambda_1,i_2)}$
if $k=2$, and otherwise $B_{k,k+1}=((a_k^{-1}A_ka_k)\times G_{k+1}^\partial)\cap W_{(\lambda_k,i_{k+1})}$; in any case, the supplied assumptions
imply that $B_{k,k+1}$ is a finitely generated subgroup of $G_k\times G_{k+1}^\partial$ and thus its second projection $A_{k+1}$ is finitely generated
in $G_{k+1}$, too. 

For (2), the effective versions of the relative Howson properties yield that at each stage of the described process where an intersection of two finitely
generated subgroups (with given finite generating sets) is taken, there is an algorithm computing a finite generating set for that intersection. 
Composing these algorithms, this means that there is an algorithm which, presented with a finite generating set for $K\leq G_1$, computes a finite
generating set for $(K,\mathbf{x},\mathbf{x})\theta$.
\end{proof}

\begin{cor}
Let $\cE$ be a finite biordered sets such that for any two regular $\D$-classes $D_1,D_2$ of $\ig{\cE}$ (with maximal subgroups $G_1,G_2$, 
respectively), each vertex group $W=W_{(\lambda,i)}$ of $\mathcal{A}(D_1,D_2)$ such that there exists a non-regular element of the form 
$(i',g_1,\lambda)(i,g_2,\lambda')\in\ig{\cE}$ for some $i'\in I_1$, $g_1\in G_1$, $g_2\in G_2$, $\lambda'\in\Lambda_2$, and the class 
$\mathcal{T}$ of all tall subgroups of $G_1\times G_2^\partial$, we have that the property $\mathsf{eRHP}(G_1\times G_2^\partial,W,\mathcal{T})$ 
holds. Then the \Sch groups of $\ig{\cE}$ are effectively computable, in the sense that there is an algorithm which, presented with an element 
$\mathbf{x}\in\ig{\cE}$, identifies a maximal subgroup $G$ of $\ig{\cE}$ and outputs finite generating sets of its subgroups $L\lnor K\leq G$
such that $\Gamma_{H_\mathbf{x}}\cong K/L$.
\end{cor}

\begin{rmk}\label{rmk:almost-finite-fp}
As already seen in Remark \ref{rmk:almost-finite}, the assumptions of the previous corollary are met when all non-maximal regular $\D$-classes
of $\ig{\cE}$ have finite maximal subgroups. Then $\Gamma_{H_\mathbf{x}}\cong K/L$ is either finite (for example, when $G$ is finite), or, otherwise,
when $G$ is free, $K$ is a finitely generated subgroup of $G$, so it is also free (and finitely generated). Since $L$ is then a finitely
generated normal subgroup of $K$, it follows that for the considered finite biorders $\cE$, each \Sch group of $\ig{\cE}$ must be finitely presented.
\end{rmk}

With our third objective in mind, of determining the ``shape'' of arbitrary $\D$-classes in $\ig{\cE}$, but also with the aim of elucidating
the left-right duality aspect of the results proved thus far, we now introduce the function $\ol\theta$ mapping subsets of $G_m$ to subsets of $G_1$ 
that is (in a definite sense that will become clear shortly) dual to $\theta$. So, once again, let $\mathbf{u}=(i_1,a_1,\lambda_1)\dots
(i_m,a_m,\lambda_m)$ and $\mathbf{v}=(j_1,b_1,\mu_1)\dots(j_m,b_m,\mu_m)$ be two elements of $\ig{\cE}$, both with $\D$-fingerprint 
$(D_1,\dots,D_m)$, and let $B\subseteq G_m$. We construct a sequence of sets $B_k\subseteq G_{m+1-k}$, $1\leq k\leq m$, starting from $B_1=B$. We let 
$$
C_{m-1,m} = (G_{m-1}\times B)\cap W_{(\lambda_{m-1},i_m)}h_{m-1,m}
$$
where $h_{m-1,m}$ is the label of some fixed walk from $(\lambda_{m-1},i_m)$ to $(\mu_{m-1},j_m)$ in $\mathcal{A}(D_{m-1},D_m)$ (if such walk exists,
otherwise $C_{m-1,m}=\es$). Then we set
$B_2=C_{m-1,m}\pi_1$ where $\pi_1$ is the first projection map. Assuming that $B_k$ has been constructed for some $1<k<m$, let
$$
C_{m-k,m+1-k} = (G_{m-k}\times b_{m+1-k}^{-1}B_k^{-1}a_{m+1-k}) \cap W_{(\lambda_{m-k},i_{m+1-k})}h_{m-k,m+1-k}
$$
where $h_{m-k,m+1-k}$ is the label of some fixed walk between $(\lambda_{m-k},i_{m+1-k})$ and $(\mu_{m-k},j_{m+1-k})$ in the automaton/graph 
$\mathcal{A}(D_{m-k},D_{m+1-k})$ (if any, otherwise we have $C_{m-k,m+1-k}=\es$). Now we define $B_{k+1}=C_{m-k,m+1-k}\pi_1$. Finally, the value of 
$\ol\theta(\mathbf{u},\mathbf{v},B)$ is the set $B_m\subseteq G_1$ obtained at the end of this process.

Now most of the results stated and proved in the last two sections of this paper can be rephrased in terms of the function $\ol\theta$ and
proved by left-right dual arguments to the ones already presented. Here we provide, without proof, a summary of the most relevant of these dual results
(here we use, without any particular reference, the notation from the pertaining original statements):
\begin{itemize}
\item For any right coset $Lx$ of a subgroup $L\leq G_m$, $\ol\theta(\mathbf{u},\mathbf{v},Lx)$ is either empty, or a left coset of a subgroup of $G_1$.
\item The duals of statements (2) and (3) of Theorem \ref{thm:theta} hold under the (effective) relative Howson property assumptions; however, the class
$\mathcal{T}_k$ of tall subgroups of $G_k\times G_{k+1}^\partial$ should be replaced by the class $\mathcal{W}_k$ of \emph{wide} subgroups of the form 
$G_k\times H$, where $H$ is a finitely generated subgroup of $G_{k+1}^\partial$ (and thus of $G_{k+1}$).
\item $\mathbf{u}=\mathbf{v}$ holds if and only if $a_1^{-1}b_1\in\ol\theta(\mathbf{u},\mathbf{v},\{b_ma_m^{-1}\})$.
\item $\mathbf{x}\;\R\;\mathbf{y}$ if and only if $i_1=j_1$ and $a_1^{-1}b_1\in\ol\theta(\mathbf{x},\mathbf{y},G_m)$.
\item $\mathbf{x}\;\L\;\mathbf{y}$ if and only if $\lambda_m=\mu_m$ and $\ol\theta(\mathbf{x},\mathbf{y},\{b_ma_m^{-1}\})\neq\es$.
\item $\mathbf{x}\;\D\;\mathbf{y}$ if and only if $\ol\theta(\mathbf{x},\mathbf{y},G_m)\neq\es$.
\item For any subgroup $K\leq G_m$, $\ol\theta(\mathbf{x},\mathbf{x},K)$ is a subgroup of $G_1$.
\item $\ol\theta(\mathbf{x},\mathbf{x},E)\lnor \ol\theta(\mathbf{x},\mathbf{x},G_m)$.
\item $\Gamma_{H_\mathbf{x}}\cong \ol\theta(\mathbf{x},\mathbf{x},G_m)/\ol\theta(\mathbf{x},\mathbf{x},E)$, and thus the \Sch group $\Gamma_{H_\mathbf{x}}$
is also a divisor of $G_1$.
\end{itemize}

Here is the final main result of this section.

\begin{thm}\label{thm:Dshape}
Let $\cE$ be a finite biordered set, and let 
$$\mathbf{x}=(i_1,a_1,\lambda_1)\dots(i_m,a_m,\lambda_m)\in\ig{\cE}$$ be of $\D$-fingerprint $(D_1,\dots,D_m)$. 
\begin{itemize}
\item[(1)] The number of $\R$-classes contained in $D_\mathbf{x}$ is 
$$
|I_1| \cdot \left(G_1:\ol\theta(\mathbf{x},\mathbf{x},G_m)\right).
$$
\item[(2)] The number of $\L$-classes contained in $D_\mathbf{x}$ is 
$$
|\Lambda_m| \cdot \left(G_m:(G_1,\mathbf{x},\mathbf{x})\theta\right).
$$
\item[(3)] 
\begin{align*}
|H_\mathbf{x}| &= \left((G_1,\mathbf{x},\mathbf{x})\theta:(E,\mathbf{x},\mathbf{x})\theta\right)=
                  \left(\ol\theta(\mathbf{x},\mathbf{x},G_m):\ol\theta(\mathbf{x},\mathbf{x},E)\right),\\
|R_\mathbf{x}| &= |\Lambda_m| \cdot \left(G_m:(E,\mathbf{x},\mathbf{x})\theta\right),\\
|L_\mathbf{x}| &= |I_1| \cdot \left(G_1:\ol\theta(\mathbf{x},\mathbf{x},E)\right),\\
|D_\mathbf{x}| &= |I_1| \cdot |\Lambda_m| \cdot \left(G_1:\ol\theta(\mathbf{x},\mathbf{x},E)\right)\cdot \left(G_m:(G_1,\mathbf{x},\mathbf{x})\theta\right)\\ 
               &= |I_1| \cdot |\Lambda_m| \cdot \left(G_1:\ol\theta(\mathbf{x},\mathbf{x},G_m)\right)\cdot \left(G_m:(E,\mathbf{x},\mathbf{x})\theta\right)
\end{align*}
In particular, if both groups $G_1,G_m$ are finite then the $\D$-class $D_\mathbf{x}$ is finite as well.
\end{itemize}
\end{thm}

\begin{proof}
Since (1) is left-right dual to (2), we skip the proof of the former and prove that latter result. From Theorem \ref{thm:Green} we know that 
$$
D_\mathbf{x}=\{(j,c,\lambda_1)\dots(i_m,b,\mu):\ j\in I_1, c\in G_1, b\in G_m, \mu\in\Lambda_m\}.
$$
Define a mapping $\xi:D_\mathbf{x}\to \Lambda_m\times \mathcal{R}_{G_m}((G_1,\mathbf{x},\mathbf{x})\theta)$, where $\mathcal{R}_{G_m}(K)$ denotes the
set of all right cosets of a subgroup $K$ of $G_m$, by
$$
[(j,c,\lambda_1)\dots(i_m,b,\mu)]\xi = (\mu,[(G_1,\mathbf{x},\mathbf{x})\theta]b).
$$
Clearly, this mapping is surjective. Now Corollary \ref{cor:Green} implies that two elements of $D_\mathbf{x}$ map to the same pair under $\xi$ 
if and only if they are $\L$-related. Hence, the required number of $\L$-classes is precisely the size of the image of $\xi$, and the result follows.

In (3), the cardinality of $H_\mathbf{x}$ must be the same as that of its \Sch group, thus the result follows immediately from Theorem \ref{thm:Sch}.
The cardinality of the $\R$-class $R_\mathbf{x}$ is the cardinality $|H_\mathbf{x}|$ multiplied by the number of $\L$-classes obtained in (2); the result
for $|L_\mathbf{x}|$ follows dually. Finally, the cardinality of $D_\mathbf{x}$ is both $|R_\mathbf{x}|$ multiplied by the number of $\R$-classes from
(1) and $|L_\mathbf{x}|$ multiplied by the number of $\L$-classes obtained in (2), thus we obtain (4).
\end{proof}


\section{Open problems}\label{sec:prob}

We would like to finish off the paper with a couple of questions (or, rather, two groups of questions) stemming from the presented material that might be
interesting as a subject of future research.

Firstly, we would very much like to know to which extent is Theorem \ref{thm:Sch} ``sharp''. Namely, let us recall that it states that any \Sch
group of $\ig{\cE}$ for a finite biorder $\cE$ must be a divisor of some maximal subgroup of $\ig{\cE}$; in fact, given $\mathbf{x}\in\ig{\cE}$ it 
identifies a subgroup $G$ of $\ig{\cE}$ and $L\lnor K\leq G$ such that $\Gamma_{H_\mathbf{x}}\cong K/L$. Under certain further assumptions, both
$K$ and $L$ are finitely generated, and, with the presence of further conditions, we can guarantee that the generating sets of both $K$ and $L$ are 
effectively computable from $\cE$. Our questions is, loosely speaking, whether this is all there is to know about these \Sch groups, or there is in fact 
further information about them that is yet to be learned. We also exhibit some weaker versions of this question that might be worth pursuing.

\begin{que}
\begin{itemize}
\item[(1)] Is is true that for any finitely presented group $G$ and its arbitrary divisor $K/L$ there exists a finite biordered set $\cE$ such that there exists
$\mathbf{x}=\ig{\cE}$ of $\D$-fingerprint $(D_1,\dots,D_m)$ such that the maximal subgroup of either $D_1$ or $D_m$ is isomorphic to $G$ , while 
$\Gamma_{H_\mathbf{x}}\cong K/L$?
\item[(2)] Is (1) true at least for divisors $K/L$ arising from finitely generated subgroups $L\lnor K\leq G$?
\item[(3)] Failing to answer (1) or (2), is (1) at least true when $G$ is a finite group?
\end{itemize}
\end{que}

Let us note that it was already proved in \cite[Theorem 9.2]{DGR} that for an arbitrary finitely presented group $G$ and its arbitrary finitely generated subgroup
$K$ there exists a finite biorder $\mathcal{B}_{G,K}$ arising from a finite band (idempotent semigroup) such that there is an element $\mathbf{x}\in\ig{\mathcal{B}_{G,K}}$
of $\D$-fingerprint $(D_1,D_2)$, where $D_1,D_2$ are distinct $\D$-classes of $\mathcal{B}_{G,K}$, both with maximal subgroup isomorphic to $G$, so that 
$\Gamma_{H_\mathbf{x}}\cong K$. In other words, the answer to the question (2) above is ``yes'' when $L$ is the trivial subgroup.

Secondly, another direction for further investigations is spurred by the fact that under the assumption that all non-maximal regular $\D$-classes of $\ig{\cE}$
have finite maximal subgroups (see \cite[Theorem 6.1]{YDG2} and the passage preceding it), the word problem $\ig{\cE}$ if decidable, and, furthermore, as seen
in the paper, practically all ``structural parameters'' of $\ig{\cE}$ (such as the contact automata, vertex groups, the functions $\theta,\ol\theta$, Green's
relations, the \Sch groups, the shapes and sizes of arbitrary $\D$-classes, etc.) are effectively computable. As shown in \cite{GR2,Do2}, this is, for example, 
the case for biordered sets of both the full transformation monoid $T_n$ and of the full partial transformation monoid $PT_n$ on an $n$-element set, a couple of
very natural examples. In this sense we are posing the following problem.

\begin{prb}
Describe completely the structure of $\ig{\cE_{T_n}}$ and $\ig{\cE_{PT_n}}$ and provide explicit solutions for their word problems. In particular, describe 
the vertex groups for their contact automata/graphs, their Green's relations, and all the enumerative parameters of their $\D$-classes.
\end{prb}

Note that the above question for $\ig{\cE_{T_n}}$ is merely a restatement of the first sentence of \cite[Section 9]{GR2}; however, the results accumulated since 
the publication of \cite{GR2} (this paper included) make the achievement of this goal much more realistic than before.

\begin{ackn}
The author is grateful to the anonymous referee for a thorough reading of the manuscript and a number of useful comments and suggestions.
\end{ackn}


\end{document}